 \documentclass[10pt,reqno]{amsart}
\usepackage{amsmath,amsthm}

\usepackage{amssymb}
\usepackage{amsfonts}
\usepackage{amscd}
\usepackage[mathscr]{eucal}
\usepackage{stmaryrd}  

\usepackage{hyperref}
\usepackage{enumerate}
\input xy
\xyoption{all}
\numberwithin{equation}{section}
\theoremstyle{definition}
\newtheorem{dfn}{Definition}[section]

\newtheorem{rem}[dfn]{Remark}
\theoremstyle{plain}
\newtheorem{thm}[dfn]{Theorem}
\newtheorem*{thmI}{Theorem I}
\newtheorem*{thmII}{Theorem II}
\newtheorem{prp}[dfn]{Proposition}
\newtheorem{cor}[dfn]{Corollary}
\newtheorem{lem}[dfn]{Lemma}
\newtheorem{prob}[dfn]{Problem}

\title[Solution of a $q$-difference Noether problem...]{Solution of a $q$-difference Noether problem and
 the quantum Gelfand-Kirillov conjecture for $\mathfrak{gl}_N$}
\author{Vyacheslav Futorny}
\author{Jonas T. Hartwig }
\address{Department of Mathematics,
 University of S\~ao Paulo,
 S\~ao Paulo, Brazil and Max Planck Institute for Mathematics, Bonn, Germany}
 \email{futorny@ime.usp.br}
\address{Department of Mathematics, Stanford University, Stanford, CA, USA}
\email{jonas.hartwig@gmail.com }
\newcommand\vi{\varphi}
\newcommand\al{\alpha}\newcommand\be{\beta} \newcommand\ga{\gamma}

 \newcommand\la{\lambda}\newcommand\La{\Lambda}

\newcommand{\Ga}{\Gamma}

\newcommand{\tmu}{\widetilde{\mu}}

\newcommand{\iv}[2]{\llbracket #1,#2 \rrbracket}  

\newcommand\mf{\mathfrak}

\newcommand\Mscr{\mathcal{M}}
\newcommand\C{\mathbb{C}}    \newcommand\Z{\mathbb{Z}}   
\newcommand\K{\Bbbk}
\newcommand\F{\mathbb{F}}

 \newcommand\Gal{\mathrm{Gal}}
\newcommand{\gl}{\mathfrak{gl}}
\newcommand{\im}{\mathrm{im}}

\DeclareMathOperator{\End}{End}

\DeclareMathOperator{\sgn}{sgn}
\DeclareMathOperator{\Frac}{Frac}
\DeclareMathOperator{\Stab}{Stab}
\DeclareMathOperator{\Specm}{Specm}
\DeclareMathOperator{\Supp}{Supp}

\DeclareMathOperator{\chara}{char}

\begin{document}

\begin{abstract}
It is shown that the $q$-difference Noether problem  for all
classical Weyl groups has a positive solution, simultaneously
generalizing well known results on multisymmetric functions of
Mattuck \cite{Mattuck1968} and Miyata \cite{Mi} in the case $q=1$,
and $q$-deforming the noncommutative Noether problem for the
symmetric group \cite{FMO}. It is also shown that
 the quantum Gelfand-Kirillov conjecture  for
$\mathfrak{gl}_N$ (for a generic $q$) follows from the positive
solution of the $q$-difference Noether problem for the Weyl group
of type $D_n$. The proof is based on the theory of Galois rings
\cite{FO}.  From here we obtain a new proof of the quantum
Gelfand-Kirillov conjecture for $\mathfrak{sl}_N$, thus recovering
the result of Fauquant-Millet \cite{FM}. Moreover, we provide an
explicit description of skew fields of fractions for quantized
$\mathfrak{gl}_N$ and $\mathfrak{sl}_N$ generalizing \cite{AD}.
\end{abstract}

\maketitle

\section{Introduction}
One important tool in the study of different noncommutative domains is a comparison
of their skew
fields of fractions. One might recall the concepts of birational
equivalence in algebraic geometry and of derived equivalence in
category theory. This makes the structure problem of division algebras very important.
Sometime the situation is especially pleasant: it
was shows by Farkas, Schofield, Snider and Stafford  \cite{FSSS} that the skew field of fractions of the
group algebra of finitely generated torsion free nilpotent group
determines the group up to isomorphism. Of course, in general the
problem is way more complicated.
As it was pointed
in \cite{FSSS} very little is known about division algebras which
are infinite dimensional over their centers. In particular, it is
very difficult to decide when two such algebras are isomorphic.

 The classical
Gelfand-Kirillov conjecture states that the skew field  of
fractions (equivalently, quotient division ring) of the universal
enveloping algebra of an algebraic Lie algebra over an
algebraically closed field of characteristic zero is isomorphic to
a Weyl field, that is, a skew field of fractions of the Weyl
algebra over a purely transcendental extension of the ground field
$\K$. This conjecture was proven by Gelfand and Kirillov~\cite{GK}
for $\mf{gl}_N$ and $\mf{sl}_N$ and for nilpotent Lie algebras.
For solvable Lie algebras the conjecture was proven independently
by Borho, Gabriel and Rentschler \cite{BGR}, Joseph \cite{Jo} and
McConnell \cite{Mc}.  Moreover, Alev, Ooms and Van den Bergh~\cite{AOV1}
proved the conjecture for all Lie algebras of dimension at most
eight. However, the same authors found counterexamples to the
conjecture for mixed Lie algebras \cite{AOV2}. Also, Premet
\cite{P} showed that the conjecture fails for orthogonal Lie
algebras and for simple Lie algebras of types $E_6,E_7,E_8$ and $F_4$.

An analogue of the Gelfand-Kirillov conjecture was shown for finite
$W$-algebras of type $A$ \cite{FMO}.

In this paper we fully solve the quantum Gelfand-Kirillov conjecture for
the quantized $\gl_N$ over $\C$.  Let $\K$ be a field, $q\in\K$ be
nonzero, $S=(s_{ij})$ be a skew-symmetric
$n\times n$ integer matrix. Define the following \emph{quantum
polynomial algebra over $\K$}:
\begin{equation} \label{eq:AqS}
\K_{q,S}[X_1,\ldots,X_n]:=\K\langle X_1,\ldots,X_n\mid
X_iX_j=q^{s_{ij}}X_jX_i\rangle.
\end{equation}
A \emph{quantum Weyl field over $\K$} is the skew field of
fractions of an algebra of the form \eqref{eq:AqS}. We will
discuss alternative definitions of quantum Weyl fields in
Section \ref{sec:qWskewfields}.

We say that a unital associative $\C$-algebra $A$ admitting a skew
field of fractions $\Frac(A)$ satisfies the \emph{quantum
Gelfand-Kirillov conjecture} if $\Frac(A)$ is isomorphic to a
quantum Weyl field over a purely transcendental field extension
$\K$ of $\C$ (cf. \cite{BG}).

The quantum Gelfand-Kirillov conjecture for
$U_q(\mathfrak g)$ has been studied for almost 20 years by many
authors. Let $\mathfrak{g}$ be any complex finite-dimensional
semi-simple Lie algebra, $\mathfrak{n}$ the nilpotent radical of a
Borel subalgebra $\mf{b}$ of $\mathfrak{g}$, and $G$ the simply
connected group associated to $\mathfrak{g}$. B. Feigin formulated
the quantum Gelfand-Kirillov conjecture at RIMS in 1992 for
$U_q(\mathfrak{n})$, which is now known as Feigin's conjecture.
For generic values of $q$, Alev, Dumas
\cite{AD}, Iohara, Malikov \cite{IM} and Joseph \cite{Jo1}
have shown that $\Frac U_q(\mathfrak{n})$
satisfies the quantum Gelfand-Kirillov conjecture, while Caldero
\cite{Ca} proved it for $\Frac U_q(\mathfrak{n})$ and $\Frac
\C_q[G]$. Panov \cite{Pa} has proved that $U_q(\mf{b})$ (and
generalizations) also satisfy the quantum Gelfand-Kirillov
conjecture.

That the skew field of fractions of (certain extensions of)
$U_q(\mf{sl}_2)$ and
 $U_q(\mf{sl}_3)$ satisfy the quantum
Gelfand-Kirillov conjecture was proved in \cite{AD} by explicitly
calculating the skew fields. Finally, Fauquant-Millet \cite{FM}
proved the quantum Gelfand-Kirillov conjecture for
$U_q(\mf{sl}_N)$ by modifying the original proof of Gelfand and
Kirillov in the classical case.

We refer the reader to \cite{BG}, \cite{G} and references therein for a
detailed historical account of the Gelfand-Kirillov conjecture for quantized enveloping algebras.

Our contribution to the quantum Gelfand-Kirillov conjecture
consists of explicit calculation of the skew fields for
$U_q(\mf{gl}_N)$ and (certain extension of) $U_q(\mf{sl}_N)$ which provides a new proof
for the conjecture in these cases. In particular, we recover the
results of Alev and Dumas \cite{AD}.

Let
$\mathcal{O}_q(\K^2)$ 
 denotes the \emph{quantum plane} $\K\langle x,y\mid
yx=qxy\rangle$ over a field $\K$, $\bar{q}=(q_1, \ldots, q_n)$ a
tuple of nonzero elements of $\K$. Let $n$ be a positive integer
and $\mathcal{O}_{\bar{q}}(\K^{2n})$ a quantum affine space:
\begin{gather}
\mathcal{O}_{\bar{q}}(\K^{2n}):=\mathcal{O}_{q_1}(\K^2)\otimes_\K
\mathcal{O}_{q_2}(\K^2)\otimes_\K\cdots\otimes_\K \mathcal{O}_{q_n}(\K^2) \\
\simeq \K\langle x_1,\ldots, x_n, y_1,\ldots,y_n\mid
y_ix_j=q^{\delta_{ij}}x_jy_i,\,
 [x_i,x_j]=[y_i,y_j]=0, \,\forall i,j\in\iv{1}{n}\rangle.
\nonumber
\end{gather}
When $q_1=\ldots=q_n=q$ then we simply denote
$\mathcal{O}_{\bar q}(\K^{2n})$ by $\mathcal{O}_q(\K^{2n})$.

We show

\begin{thmI}
The quantum Gelfand-Kirillov conjecture holds for $U_q(\mf{gl}_N)$
for $q\in \C$ not a root of unity. Explicitly, there exists a
$\C$-algebra isomorphism
\begin{equation}
\Frac\big(U_q(\mf{gl}_N)\big) \simeq
 \Frac \Big( \mathcal{O}_q(\K^2)^{\otimes_\K (N-1)}\otimes_\K
 \mathcal{O}_{q^2}(\K^2)^{\otimes_\K (N-1)(N-2)/2} \Big),
\end{equation}
where $\K$ denotes the field $\C(Z_1,\ldots,Z_N)$ of rational
functions in $N$ variables over $\C$.
\end{thmI}

The proof is based on the reduction of the quantum Gelfand-Kirillov
conjecture to the \emph{$q$-difference Noether problem} for the
Weyl group of type $D_n$.

%
%
Let $W_n=W(X_n)$ be the Weyl group of type $X_n$ where $X\in\{A,B,C,D\}$.
The group $W_n$ acts naturally on
$\mathcal{O}_q(\K^{2n})$ by $\K$-algebra automorphisms
(see Section \ref{section-weyl} for details).
Let $\mathcal{F}_{q,n}$ (respectively $\mathcal{F}_{\bar q,n}$)
denote the skew field of fractions of
$\mathcal{O}_q(\K^{2n})$ (respectively $\mathcal{O}_{\bar q,n}(\K^{2n})$).
The action of $W_n$ on
$\mathcal{O}_q(\K^{2n})$ induces an action of $W_n$ on
$\mathcal{F}_{q,n}$.
 We let
\[\mathcal{F}_{q,n}^{W_n}:=\big\{a\in \mathcal{F}_{q,n}\mid w(a)=a,\,\forall w\in W_n\big\}\]
denote the subalgebra (skew subfield) of invariants under $W_n$.
Consider the following problem, which we call the
\emph{q-difference Noether problem for $W_n$}:
\begin{prob}\label{prob:qNoether}
Do there exist $q_1,\ldots,q_n\in \langle q\rangle:=\{q^k\mid
k\in\Z\}$ such that
\begin{equation}\label{eq:qNoether}
\mathcal{F}_{q,n}^{W_n}\simeq \mathcal{F}_{\bar q,n},
\end{equation}
where ${\bar q}=(q_1, \ldots, q_n)$,
as $\K$-algebras?
\end{prob}

We answer this question affirmatively and prove our main result:

\begin{thmII} The $q$-difference Noether problem for
the group $W_n$ has a positive solution, namely
\begin{equation}
 \mathcal{F}_{q,n}^{W_n}\simeq  \mathcal{F}_{\bar q,n},
\end{equation}
where
\[\bar q=
\begin{cases}
(q,q,\ldots,q),& \text{if $W_n=W(A_n)=S_n$},\\
(q^2,q^2,\ldots,q^2),& \text{if $W_n=W(B_n)=W(C_n)$},\\
(q,q^2,q^2,\ldots,q^2),&\text{if $W_n=W(D_n)$}.
\end{cases}
\]
\end{thmII}
This can be viewed as quantum
versions of classical results of Mattuck \cite{Mattuck1968} and of
Miyata \cite{Mi}.

As a corollary we get an isomorphism of $\K$-algebras
$$\big(\Frac (A_1^q(\K)^{\otimes_\K n})\big)^{S_n}\simeq \Frac
(A_1^q(\K)^{\otimes_\K n}),$$ where $A_1^q(\K):=\K\langle x,y\mid
yx-qxy=1\rangle$ (see Corollary~\ref{cor-weyl}). This result
 can be regarded as a $q$-deformation of the isomorphism
$\Frac(A_n(\K))^{S_n}\simeq \Frac(A_n(\K))$ proved in \cite{FMO}.
Here $A_n(\K)$ is the $n$:th Weyl algebra over $\K$.

Our proof of the quantum Gelfand-Kirillov conjecture relies on
the theory of Galois rings \cite{FO}. Using this theory and
Gelfand-Tsetlin representations constructed by Mazorchuk and
Turowska \cite{MT} we show that $U_q(\mf{gl}_N)$ can be embedded
into the $(W_1\times W_2\times\cdots\times W_N)$-invariants of a certain skew group ring
(Theorem~\ref{thm:UqGalois}), where $W_m$ is the Weyl group of
type $D_m$. Using this realization of $U_q(\mf{gl}_N)$ the problem
is then reduced to computation of the skew field of the
$W_m$-invariants in the tensor product of $m$ quantum planes. This
computation follows from positive solution of the $q$-difference
Noether problem for the Weyl group $W_m$.

\section{Preliminaries}

\subsection{Notation}
Unless otherwise stated,
the ground field $\K$ is arbitrary and $q\in\K$ is only assumed to be nonzero.
All rings and algebras will be understood to be associative and unital.
By a \emph{skew field} we mean a division ring. The skew field of fractions,
provided it exists, of an algebra $A$ will be denoted by $\Frac(A)$.
A well-known fact (see for example \cite[Sec.~3.2.1]{D})
is that if $A$ is an Ore domain, acted upon by a finite group $G$ with $|G|$ invertible in $A$, then
the invariants $A^G:=\{a\in A\mid g(a)=a,\,\forall g\in G\}$
is also an Ore domain and $\Frac(A^G)=\Frac(A)^G$.
We will use the generalized Kronecker delta notation $\delta_P$ for a statement $P$,
defined by
\begin{equation}\label{eq:kronecker}
\delta_P=\begin{cases}1,&\text{if $P$ is true,}\\0,&\text{otherwise.}\end{cases}
\end{equation}
For $a,b\in\Z$ we use the notation $\iv{a}{b}=\{x\in\Z\mid a\le x\le b\}$.
If a group $G$ acts on a ring $R$ by automorphisms, we denote the corresponding skew group ring by $R\ast G$.
We sometimes use the $q$-commutator notation $[a,b]_q=ab-qba$.

\subsection{The algebra \texorpdfstring{$U_q(\mathfrak{gl}_N)$}{Uq(glN)}}
Assume $q^2\neq 1$. For positive integers $N$ we let $U_N=U_q(\mathfrak{gl}_N)$ denote the unital associative $\K$-algebra with generators $E_i^\pm$, $K_j, K_j^{-1}$, $i\in\iv{1}{N-1}$, $j\in\iv{1}{N}$ and relations \cite[p.163]{KS}
\begin{gather*}
K_iK_i^{-1}=K_i^{-1}K_i=1, \quad [K_i,K_j]=0,\quad\forall i,j\in\iv{1}{N},\\
\begin{aligned}
K_iE_j^\pm K_i^{-1} &= q^{\pm(\delta_{ij}-\delta_{i,j+1})}E_j^{\pm}, \quad\forall i\in\iv{1}{N}, \forall j\in\iv{1}{N-1},\\
[E_i^+,E_j^-]&=\delta_{ij}\frac{K_iK_{i+1}^{-1}-K_{i+1}K_i^{-1}}{q-q^{-1}}, \quad\forall i,j\in\iv{1}{N-1},\\
[E_i^{\pm},E_j^{\pm}]&=0,\quad |i-j|>1,
\end{aligned}\\
(E_i^\pm)^2E_j^\pm -(q+q^{-1})E_i^\pm E_j^\pm E_i^\pm + E_j^\pm (E_i^\pm)^2 =0, \quad |i-j|=1.
\end{gather*}

\subsection{Quantum Weyl fields} \label{sec:qWskewfields}
If $n$ is a positive integer, the Weyl algebra $A_n(\K)$ is the
algebra of  differential operators on polynomial ring
$\mathcal{O}(\Bbbk^{n})$. This algebra is a simple Noetherian
domain which allows a skew field of fractions called a \emph{Weyl
field}. In this section we recall some well-known results
regarding the $q$-analogue of Weyl fields.

Recall the quantum polynomial algebra \eqref{eq:AqS}:
\begin{equation}
\K_{q,\left[\begin{smallmatrix}0&1\\-1&0\end{smallmatrix}\right]}[X_1,X_2]\simeq \mathcal{O}_q(\Bbbk^2).
\end{equation}

\begin{prp}\label{prp:normalform}
Let $n$ be a positive integer.
Let $S$ be a $2n\times 2n$ skew-symmetric integer matrix. Then there exist integers $k_1,\ldots,k_n$ and an algebra isomorphism
\begin{equation}
\Frac\big(\K_{q,S}[X_1,\ldots,X_{2n}]\big) \simeq
\Frac\big(\mathcal{O}_{q^{k_1}}(\Bbbk^2)\otimes\cdots\otimes
\mathcal{O}_{q^{k_n}}(\Bbbk^2)\big).
\end{equation}
\end{prp}
\begin{proof}
Similar to the proof of \cite[Theorem 4.8]{H}, but we provide
details for convenience. Denote
$\K_q[x,y]=\mathcal{O}_{q}(\Bbbk^2)$. It is enough to show that
the corresponding Laurent analogs, $\K_{q,S}[X_1^{\pm
1},\ldots,X_{2n}^{\pm 1}]$ and
$$\K_{q^{k_1}}[x^{\pm 1},y^{\pm 1}]\otimes\cdots
\otimes \K_{q^{k_n}}[x^{\pm 1},y^{\pm 1}]$$ are isomorphic.
Consider a change of generators
\[X_i':=X_1^{u_{1i}}\cdots X_{2n}^{u_{2n,i}},\qquad i=1,\ldots,2n,\]
where $U=(u_{ij})$ is an invertible $2n\times 2n$ integer matrix.
The new commutation relations are
\begin{equation}\label{eq:Xprimrel}
X_i'X_j'=q^{s'_{ij}}X_j'X_i',\qquad i,j=1,\ldots,2n,
\end{equation}
where $s'_{ij}$ are the entries of the matrix $S':=U^tSU$. By
Theorem IV.1 in \cite{N} there is an invertible $2n\times 2n$
integer matrix $U$ such that $U^tSU$ is block diagonal with
skew-symmetric $2\times 2$ blocks on the diagonal. That is,
\begin{equation}\label{eq:UtsU}
U^tSU=\bigoplus_{i=1}^n \begin{bmatrix}0& k_i\\ -k_i& 0\end{bmatrix}
\end{equation}
for some $k_i\in\Z$. Put $x_i=X_{2i}'$ and $y_i=X_{2i-1}'$ for
$i=1,\ldots,n$. Then \eqref{eq:Xprimrel} and \eqref{eq:UtsU} imply
that $y_ix_i=q^{k_i}x_iy_i$  for all
$i$ and $[x_i,x_j]=[x_i,y_j]=[y_i,y_j]=0$ for all $i\neq j$. Thus
there is a $\K$-algebra isomorphism
\begin{align*}
\K_{q,S}[X_1^{\pm 1},\ldots,X_{2n}^{\pm 1}]&\overset{\sim}{\longrightarrow} \K_{q^{k_1}}[x^{\pm 1},y^{\pm 1}]\otimes\cdots\otimes \K_{q^{k_n}}[x^{\pm 1},y^{\pm 1}],\\
\intertext{determined by}
x_i&\longmapsto 1^{\otimes i-1} \otimes x \otimes 1^{\otimes n-i},\\
y_i&\longmapsto 1^{\otimes i-1} \otimes y \otimes 1^{\otimes n-i}.
\end{align*}
\end{proof}

Let $\bar q=(q_1,\ldots,q_n)\in(\K\backslash\{0\})^n$ and
$\La=(\la_{ij})$ be an $n\times n$ matrix with $\la_{ij}\in\K,
\la_{ij}\la_{ji}=\la_{ii}=1$ for all $i,j$. The
\emph{multiparameter quantized Weyl algebra} $A_n^{\bar q,\La}(\K)$
 was introduced by Maltsiniotis \cite{Mal} (see also
\cite{J}). This algebra can be viewed as algebra of $q$-difference
operators on quantum affine space $\mathcal{O}_q(\Bbbk^{n})$. It
is defined as the associative unital $\K$-algebra generated by
$x_1,\ldots,x_n$ and $y_1,\ldots,y_n$ with defining relations
\begin{subequations}\label{eq:multiqweyl}
\begin{align}
y_iy_j&=\la_{ij}y_jy_i,\quad \forall i,j\\
x_ix_j&=q_i\la_{ij}x_jx_i,\quad i<j\\
x_iy_j&=\la_{ji} y_jx_i,\quad i<j\\
x_iy_j&=q_j\la_{ji}y_jx_i,\quad i>j\\
x_iy_i-q_iy_ix_i&=1+\sum_{1\le k\le i-1} (q_k-1)y_kx_k.
\end{align}
\end{subequations}

The following proposition is well-known (see for example \cite{BG} and references therein), but we provide a proof containing the explicit isomorphisms which are not always given in the literature.
\begin{prp}\label{prp:qWeyl_iso}
Let $n$ be a positive integer, $\K$ a field, and $(q_1,\ldots,q_n)\in(\K\setminus\{0,1\})^n$. Then the skew fields of fractions of the following three algebras are isomorphic:
\begin{enumerate}[{\rm (i)}]
\item The tensor product of quantum Weyl algebras
\begin{equation}\label{eq:qweylk1kn}
A_1^{q_1}(\K)\otimes_\K \cdots \otimes_\K A_1^{q_n}(\K);
\end{equation}
 \item The tensor product of quantum planes
\begin{equation}\label{eq:qplanes}
\mathcal{O}_{q_1}(\Bbbk^2)\otimes_\K\cdots\otimes_\K
\mathcal{O}_{q_n}(\Bbbk^2);
\end{equation}
 \item The multiparameter quantized Weyl algebra
\begin{equation}\label{eq:multiqweyl2}
 A_n^{\bar q,\La}(\K)
\end{equation}
with parameters $\bar q = (q_1,\ldots,q_n)$, and $\La=(\la_{ij})$, $\la_{ij}=1$ for all $i,j=1,\ldots,n$.
\end{enumerate}
\end{prp}
\begin{proof}
That \eqref{eq:qweylk1kn} and \eqref{eq:qplanes} have isomorphic skew
fields of fractions follows from the fact that there is an isomorphism
\begin{align*}
\K\langle x^{\pm 1},y\mid yx-qxy=1\rangle &\longrightarrow
\K\langle x^{\pm 1},y\mid yx=qxy\rangle\\
x &\longmapsto x\\
y &\longmapsto (qx-x)^{-1}(y-1).
\end{align*}
This is straightforward to check directly.
(One can understand this isomorphism as coming from the realization of $y$ in the left hand side as the $q$-difference operator $f(x)\mapsto \frac{f(qx)-f(x)}{qx-x}$ for $f(x)\in\K[x,x^{-1}]$ while in the right hand side $y$ can be realized as the $q$-shift operator $f(x)\mapsto f(qx)$.)

Concerning the multiparameter quantized Weyl algebra, the proof
can be derived from \cite{J}. We recall from \cite{J} that the
elements $z_i\in A_n^{\bar q,\La}(\K)$ defined by
\begin{equation}\label{eq:zidef}
z_i:=[x_i,y_i]=1+\sum_{1\le k\le i} (q_k-1)y_kx_k,\quad i=1,\ldots,n
\end{equation}
satisfy
\begin{subequations}\label{eq:ziyirels}
\begin{align}
z_iz_j &=z_jz_i,\quad\forall i,j\\
z_jy_i &= \begin{cases}y_iz_j,&j<i,\\ q_iy_iz_j,& j\ge i.\end{cases}
\end{align}
\end{subequations}
In $\Frac\big(A_n^{\bar q,\La}(\K)\big)$, putting
\begin{equation}\label{eq:zj'def}
z_j':=z_j\cdot z_{j-1}^{-1},\quad\forall j=1,\ldots,n,
\end{equation}
where $z_0:=1$, relations \eqref{eq:ziyirels} imply that
\begin{subequations}\label{eq:zi'yirels}
\begin{align}
z_i'z_j'&=z_j'z_i'\\
z_j'y_i&= \begin{cases}y_iz_j',&i\neq j,\\ q_i y_i z_j',&i=j.\end{cases}
\end{align}
Since $\la_{ij}=1$ for all $i,j$, \eqref{eq:multiqweyl} implies
\begin{equation}
y_iy_j=y_jy_i.
\end{equation}
\end{subequations}
Relations \eqref{eq:zi'yirels} prove that, there is a $\K$-algebra homomorphism
\begin{align*}
\Frac\big(\mathcal{O}_{q_1}(\Bbbk^2)\otimes_\K\cdots\otimes_\K
\mathcal{O}_{q_n}(\Bbbk^2)\big)&\longrightarrow
\Frac\big(A_n^{\bar q,\La}(\K)\big),\\
1^{\otimes i-1}\otimes x\otimes 1^{\otimes n-i}  &\longmapsto y_i,\\
1^{\otimes i-1}\otimes y\otimes 1^{\otimes n-i}  &\longmapsto
z_i',
\end{align*}
$\bar q=(q_1, \ldots, q_n)$. It is injective since the domain is a
skew field and surjective since in $\Frac\big(A_n^{\bar
q,\La}(\K)\big)$ we have by \eqref{eq:zidef},\eqref{eq:zj'def}
\begin{equation}
x_i=\frac{y_i^{-1}(z_i-z_{i-1})}{q_i-1}=
\frac{y_i^{-1}\Big(\prod_{j=1}^iz_j'-\prod_{j=1}^{i-1}z_j'\Big)}{q_i-1},\quad\forall i=1,\ldots,n,
\end{equation}
where $z_0:=1$.
\end{proof}

\begin{rem}
In \cite[Thm~3.5]{AD} it is proved that if $q_i, \la_{ij}$ ($i,j=1,\ldots,n$) are powers of some fixed non-root of unity $q\in\K\setminus\{0\}$, then $\Frac\big(A_n^{\bar q,\La}(\K)\big)$ is isomorphic to a quantum Weyl field $\Frac\big(\K_{q,S}[X_1,\ldots,X_{2n}]\big)$ for some $2n\times 2n$ skew-symmetric integer matrix $S$
(see also \cite[Sec~5]{P}). Combining this with Proposition \ref{prp:normalform} we get the following result.
\begin{cor}\label{cor:multiqweyl}
If all parameters $q_i, \la_{ij}$ ($i,j=1,\ldots,n$) are powers of some fixed non-root of unity $q\in\K\setminus\{0\}$, then there exists a tuple $(k_1,\ldots,k_n)\in\Z^n$ such that
\begin{equation}
\Frac\big(A_n^{\bar q,\La}(\K)\big)
\simeq\Frac\big(\mathcal{O}_{q^{k_1}}(\Bbbk^2)\otimes_\K\cdots\otimes_\K\mathcal{O}_{q^{k_n}}(\Bbbk^2)\big).
\end{equation}
\end{cor}
In general, however, the integers $k_i$ occuring in Corollary
\ref{cor:multiqweyl} require some work to determine.
\end{rem}

\section{The \texorpdfstring{$q$}{q}-difference Noether problem for \texorpdfstring{$S_n$}{Sn}}\label{section-q-noether}
Let $n$ be a positive integer.
Throughout this section, $\K$ denotes a field of characteristic zero,
and $q$ is any nonzero element of $\K$.
Let
$$\K_q[\bar x,\bar y]=\K_q[x_1,y_1]\otimes_\K
\K_q[x_2,y_2]\otimes_\K\cdots\otimes_\K \K_q[x_n,y_n]\simeq
\mathcal{O}_q(\Bbbk^{2n}),$$
 $\K_q(\bar x,\bar y)$ be the skew field of fractions of $\K_q[\bar
x,\bar y]$ and
$\K_q(\bar x,\bar y)^{S_n}$ the subalgebra of $S_n$ invariants.

\subsection{Generators and relations for the skew field of invariants}
In this section we provide a set of generators and relations for the algebra of
invariants $\K_q(\bar x,\bar y)^{S_n}$. Let
\begin{equation}
C_n^q:=\K(x_1,\ldots,x_n)\langle y_1,\ldots,y_n\rangle
\end{equation}
denote the $\K(x_1,\ldots,x_n)$-subring of $\K_q(\bar x,\bar y)$ generated by $\{y_1,\ldots,y_n\}$.
Note that $C_n^q$ is an $S_n$-invariant subspace of $\K_q(\bar x,\bar y)$ and that $\Frac(C_n^q)=\K_q(\bar x,\bar y)$.
Inspired by \cite{Mattuck1968}, we observe that the Vandermonde matrix
\begin{equation}\label{eq:Mvandermonde}
\begin{bmatrix}
1 & x_1 & x_1 & \cdots & x_1^{n-1} \\
1 & x_2 & x_2 & \cdots & x_2^{n-1} \\
  &\vdots  &       &  \ddots      & \vdots \\
1 & x_n & x_n & \cdots & x_n^{n-1}
\end{bmatrix}
\end{equation}
is invertible and thus the system of equations
\begin{equation}\label{eq:tsystem}
t_1+ x_i t_2 + x_i^2 t_3 + \cdots + x_i^{n-1} t_n = y_i,\qquad i=1,\ldots,n
\end{equation}
has a unique solution $(t_1,\ldots,t_n) \in (C_n^q)^n$.
Since the system \eqref{eq:tsystem} is $S_n$-invariant,
\begin{equation}
t_i\in (C_n^q)^{S_n},\qquad\forall i=1,\ldots,n.
\end{equation}
The explicit inverse of the matrix \eqref{eq:Mvandermonde} is well-known and implies the following description of the $t_i$.
If we introduce the generating function $P(X)\in C_n^q[X]$ by
\begin{equation}\label{eq:PXa}
P(X)=\sum_{j=1}^n t_j X^{j-1},
\end{equation}
then
\begin{equation}\label{eq:PXb}
P(X)=\sum_{j=1}^n \left(\prod_{k\in\{1,\ldots,n\}\setminus\{j\}} \frac{X-x_k}{x_j-x_k}\right)y_j.
\end{equation}
Explicitly,
\begin{equation}\label{eq:PX}
t_i = \sum_{j=1}^n \left(\frac{(-1)^{n-i}e_{n-i}'(x_1,\ldots,\widehat{x_j},\ldots,x_n)}{\prod_{k\in\{1,\ldots,n\}\setminus\{j\}}(x_j-x_k)}\right)y_j
\end{equation}
where $e_i'$ is the degree $i$ elementary symmetric polynomial in $n-1$ variables, $e_0':=1$, and $\widehat{x_j}$ means that variable should be omitted.

Since the $t_i$ and $y_i$ can be expressed through each other via
 \eqref{eq:tsystem} and \eqref{eq:PX} we have
\begin{equation} \label{eq:tgenerates}
  C_n^q=\K(x_1,\ldots,x_n)\langle t_1,\ldots,t_n \rangle,
\end{equation}
i.e. $C_n^q$ is generated as a $\K(x_1,\ldots,x_n)$-ring by $t_1,\ldots,t_n$.

\begin{prp}\label{prp:tcommutes}  For any $i,j\in\{1,\ldots,n\}$ we have
\begin{equation}
[t_i,t_j]=0.
\end{equation}
\end{prp}
The proof of Proposition~\ref{prp:tcommutes} will be given in the Appendix.

We need the following preliminary observation of the commutation relations between $t_i$ and
rational functions of $x_1,\ldots,x_n$.
\begin{lem}\label{lem:tia}
For any $a\in \K(x_1,\ldots,x_n)$ and any $i\in\{1,\ldots,n\}$ there
are $a_{i1},\ldots,a_{in}\in\K(x_1,\ldots,x_n)$ with
\begin{equation}\label{eq:tia}
t_i a = a_{i1} t_1+\cdots + a_{in}t_n.
\end{equation}
\end{lem}
\begin{proof} From \eqref{eq:PX} we know that
\[t_i=b_{i1}y_1+\cdots+b_{in}y_n\]
for some $b_{ij}\in\K(x_1,\ldots,x_n)$. Using the commutation relation $y_jx_k=q^{\delta_{jk}}x_ky_j$
we obtain that
\[t_ia=c_{i1}y_1+\cdots+c_{in}y_n\]
for some $c_{ij}\in\K(x_1,\ldots,x_n)$. Now use \eqref{eq:tsystem} to
obtain \eqref{eq:tia} for some $a_{ij}$.
\end{proof}


Combining \eqref{eq:tgenerates}, Proposition \ref{prp:tcommutes} and
Lemma \ref{lem:tia} we obtain the following result.
\begin{prp}\label{prp:leftmodule}
The set
\[
\big\{t_1^{k_1}\cdots t_n^{k_n}\mid k_1,\ldots,k_n\in\Z_{\ge 0}\big\}
\]
spans $C_n^q$ as a left $\K(x_1,\ldots,x_n)$-module.
\end{prp}

We can now prove the following statement about the generators of the invariants of $C_n^q$.
\begin{prp}
The algebra $(C_n^q)^{S_n}$ is generated as a $\K(x_1,\ldots,x_n)^{S_n}$-ring by $\{t_1,\ldots,t_n\}$.
\end{prp}
\begin{proof}
Let $u\in(C_n^q)^{S_n}$. By Proposition \ref{prp:leftmodule} we have
\[u=\sum_{k\in(\Z_{\ge 0})^n} u_k t_1^{k_1}\cdots t_n^{k_n}\]
for some $u_k\in \K(x_1,\ldots,x_n)$.
Since $u$ and $t_1,\ldots,t_n$ are $S_n$-fixed we have
\[u=\frac{1}{|S_n|}\sum_{w\in S_n} w(u)=\sum_{k\in(\Z_{\ge 0})^n} \Big(\frac{1}{|S_n|}\sum_{w\in S_n} w(u_k)\Big) t_1^{k_1}\cdots t_n^{k_n}\]
which proves that $u\in \K(x_1,\ldots,x_n)^{S_n}\langle t_1,\ldots,t_n\rangle$.
\end{proof}

As a corollary we obtain a set of generators for the skew field $\K_q(\bar x,\bar y)^{S_n}$.
\begin{cor}\label{cor:generators_of_skewfield_invariants}
$\K_q(\bar x,\bar y)^{S_n}$ is generated as a skew field over $\K$ by the following set of $2n$ elements:
\[\{t_1,\ldots,t_n\}\cup\{e_1,\ldots,e_n\}\]
where
\begin{equation}\label{eq:ed_def}
e_d:=\sum_{1\le i_1<\cdots <i_d\le n} x_{i_1}\cdots x_{i_d},\quad d\in\iv{0}{n}
\end{equation}
is the degree $d$ elementary symmetric polynomial in $x_1,\ldots,x_n$.
\end{cor}

In order to describe precise commutation relations between $t_j$ and $e_k$, it will be useful to rewrite $t_j$ as follows.
\begin{lem}
We have the following formula for $t_j$:
\begin{equation}\label{eq:formula_for_tj}
t_j=(-1)^{j-1}\Delta^{-1} \sum_{w\in S_n}\sgn(w)w\big(x_1^{n-2}x_2^{n-3}\cdots x_{n-2} e_{n-j}'y_n\big),
\quad\forall j\in\iv{1}{n},
\end{equation}
where $e_d'$ denotes the degree $d$ elementary symmetric polynomial in the variables $x_1,\ldots,x_{n-1}$ and $\Delta=\prod_{1\le i< j\le n} (x_i-x_j)$.
\end{lem}
\begin{proof}
Let $\Delta'=\prod_{1\le i<j\le n-1} (x_i-x_j)$. Let $\mathrm{Coeff}_{X^j} A(X)$ denote the coefficient of $X^j$ in a polynomial $A(X)$. By \eqref{eq:PXa} and \eqref{eq:PXb} we have, for any $j\in\{1,\ldots,n\}$,
\begin{align*}
t_j&=\sum_{i=1}^n \Big(\mathrm{Coeff}_{X^{j-1}}\prod_{k\in\{1,\ldots,n\}\setminus\{i\}} \frac{X-x_k}{x_i-x_k} \Big)y_i\\
&=\sum_{w\in S_n/S_{n-1}} w\bigg(\frac{(-1)^{n-j}e_{n-j}'}{\prod_{k=1}^{n-1} (x_n-x_k)} y_n\bigg).
\end{align*}
Here we mean that $w$ runs through a set of representatives of $S_n/S_{n-1}$.
Since $\Delta/\Delta'=\prod_{k=1}^{n-1}(x_k-x_n)$ and $w(\Delta)=\sgn(w)\Delta$ for all $w\in S_n$, we get
\begin{equation}\label{eq:tj_formula_proof}
t_j=(-1)^{j-1}\Delta^{-1} \sum_{w\in S_n/S_{n-1}} \sgn(w)\cdot w\big(e_{n-j}'\Delta' y_n\big).
\end{equation}
Writing $\Delta'$ as a determinant gives
$\Delta'=\sum_{w\in S_{n-1}} \sgn(w)w(x_1^{n-2}x_2^{n-3}\cdots x_{n-2})$. Substituting this into \eqref{eq:tj_formula_proof} and using that $e_{n-j}'$ and $y_n$ are fixed by $S_{n-1}$,  gives
\begin{equation}
t_j=(-1)^{j-1}\Delta^{-1} \sum_{\substack{w\in S_n/S_{n-1}\\w'\in S_{n-1}}}\sgn(ww')ww'(x_1^{n-2}x_2^{n-3}\cdots x_{n-2} e_{n-j}'y_n ).
\end{equation}
Since $ww'$ runs through every element of $S_n$ exactly once when $w$ ranges over a set of representatives for $S_n/S_{n-1}$ and $w'$ runs through $S_{n-1}$ we obtain
\eqref{eq:formula_for_tj}.
\end{proof}

We now have the following proposition, describing commutation
relations between the generators $t_j$ and $e_k$.
\begin{prp}\label{prp:tj_ek}
The following relations hold in $\K_q(\bar x,\bar y)^{S_n}$:
\begin{align}
\label{eq:ti_tj}
[t_i,t_j]&=0,\qquad\forall i,j\in\iv{1}{n},\\
\label{eq:ek_el}
[e_k,e_l]&=0,\qquad\forall k,l\in\iv{0}{n},\\
\label{eq:tj_ek}
t_je_k - q^{\delta_{j+k>n}}e_kt_j &
=(q-1)\sum_{i\in\Z\setminus I(n-(j+k))}
(-1)^{i+\delta_{i<0}} e_{k+i} t_{j+i},
\end{align}
for all $ j\in\iv{1}{n}$ and $k\in\iv{0}{n}$,
where $\delta_P$ is the Kronecker delta \eqref{eq:kronecker} and for all $k\in\Z$ we put
\begin{equation}\label{eq:I_def}
I(k):=\iv{\min(0,k+1)}{\max(0,k)}=
\begin{cases}
\iv{0}{k},&k\ge 0,\\
\iv{k+1}{0},&k<0,
\end{cases}
\end{equation}
and, by convention, $t_j=0$ if $j\notin\iv{1}{n}$ and $e_k=0$ if $k\notin\iv{0}{n}$.
\end{prp}

The proof of Proposition~\ref{prp:tj_ek} will be given in
Appendix.

\subsection{Simplification of the relations}
In this section we show how to inductively change generators to simplify the relations. The final set of
relations are $q$-commutation relations, which gives a positive solution to the $q$-difference Noether problem.

We will frequently use the following telescoping sum identities.
\begin{lem} If $\{T_j\}_{j\in\Z}$ is a set of commuting elements
of an algebra with at most finitely many nonzero elements, then
for all $j,k\in\Z$ the following identities hold:
\begin{align}
\label{eq:T_identity_1}
\sum_{i\in\Z\setminus I(k-j)} (-1)^{\delta_{i<0}} T_{j+i}T_{k-i} &= -\delta_{j>k} T_jT_k,\\
\label{eq:T_identity_2}
\sum_{i\in\Z\setminus I(-1+k-j)} (-1)^{\delta_{i<0}} T_{j+i}T_{k-i} &=\delta_{j<k} T_jT_k,
\end{align}
where $I(k)$ was defined in \eqref{eq:I_def}.
\end{lem}
\begin{proof}
We prove \eqref{eq:T_identity_1}. The proof of \eqref{eq:T_identity_2} is analogous.
By shifting the index of $T_i$ we may assume that $j=0$.
If $k\ge 0$, then $I(k)=\iv{0}{k}$ so making the substitution $i\mapsto k-i$ in the left hand side
of \eqref{eq:T_identity_1}
we get the same expression except that $\delta_{i<0}$ has been replaced by
$\delta_{k-i<0}$ which equals $1-\delta_{i<0}$ for $i\notin \iv{0}{k}$. So both sides
of \eqref{eq:T_identity_1} are zero in this case.
If $k<0$, then $I(k)=\iv{1+k}{0}$. The $i=k$ term in the left hand side of \eqref{eq:T_identity_1}
equals
\begin{equation}\label{eq:T_identity_1_proof}
-T_kT_0.
\end{equation}
Removing this term from the sum gives a sum over the set $\Z\setminus \iv{k}{0}$ which can be seen to
be zero, after substituting $i\mapsto k-i$ as in the previous case. Therefore the left hand side
of \eqref{eq:T_identity_1} equals \eqref{eq:T_identity_1_proof} which in turn is equal to the right
hand side of \eqref{eq:T_identity_1}, since $k<0$.
\end{proof}

The following proposition describes the recursive process for simplifying the relations among the generators.
\begin{prp}\label{prp:inductive_procedure}
Suppose $T_1,\ldots,T_n$ and $E_0,E_1,\ldots,E_n$ are elements of some skew field $\F$ containing $\K$ such that
\begin{align}
\label{eq:TiTj_assumption}
 [T_i,T_j]&=0, \quad \forall i,j\in\iv{1}{n}, \\
\label{eq:EkEl_assumption}
 [E_k,E_l]&=0, \quad \forall k,l\in\iv{0}{n},\\
\label{eq:TjEk_assumption} T_j E_k -q^{\delta_{j+k>n}}E_kT_j &=
(q-1)\sum_{i\in\Z\setminus I(n-(j+k))} (-1)^{i+\delta_{i<0}}
E_{k+i} T_{j+i},
\end{align}
$\forall j\in\iv{1}{n}, k\in\iv{0}{n},$ where by convention
$T_j=0$ for $j\notin \iv{1}{n}$ and $E_k=0$ for $k\notin
\iv{0}{n}$. Define
\begin{align}
\label{eq:tilde_T_def}
\widetilde{T}_j &=
\begin{cases}
E_jT_1T_n-(-1)^jE_0T_{n-j}T_1-(-1)^{n-j}E_nT_{n+1-j}T_n,& j\in\iv{1}{n-1}, \\
0,&\text{otherwise,}
\end{cases}\\
\label{eq:tilde_E_def}
\widetilde{E}_k &=
\begin{cases}
 T_{k+1},& k\in\iv{0}{n-1},\\
 0,&\text{otherwise.}
\end{cases}
\end{align}
Then
\begin{align}
\label{eq:tilde_T_commute}
[\widetilde{T}_i,\widetilde{T}_j] &=0,\quad\forall i,j\in\iv{1}{n-1},\\
\label{eq:tilde_E_commute}
[\widetilde{E}_k,\widetilde{E}_l] &=0,\quad\forall k,l\in\iv{0}{n-1},
\end{align}
and, where
$\circ$ denotes the opposite multiplication
$a\circ b=ba$,
\begin{equation} \label{eq:tilde_T_tilde_E}
\widetilde{T}_j \circ \widetilde{E}_k -q^{\delta_{j+k>n-1}}\widetilde{E}_k \circ \widetilde{T}_j = (q-1)\sum_{i\in\Z\setminus I(n-1-(j+k))}
(-1)^{i+\delta_{i<0}} \widetilde{E}_{k+i} \circ \widetilde{T}_{j+i}
\end{equation}
for all $j\in\iv{1}{n-1}$, $k\in\iv{0}{n-1}$. Moreover, we have the following
alternative expression for $\widetilde{T}_j$:
\begin{equation}\label{eq:tilde_T_alt}
q\widetilde{T}_j=
T_nT_1E_j-(-1)^{j}T_1T_{n-j}E_0-(-1)^{n-j}T_nT_{n+1-j}E_n,
\quad\forall j\in\iv{1}{n-1}
\end{equation}
which is equal to the right hand side of \eqref{eq:tilde_T_def}
calculated in the opposite algebra. Furthermore, the set
$\{E_0,E_n\}\cup\{\widetilde{T}_j\}_{j=1}^{n-1}\cup\{\widetilde{E}_k\}_{k=0}^{n-1}$
generates the same skew subfield of $\F$ as the original generators
$\{T_j\}_{j=1}^n\cup\{E_k\}_{k=0}^n$.
\end{prp}

The proof of Proposition~\ref{prp:inductive_procedure} will be
given in the Appendix.

We can now prove the following theorem which implies Theorem~II for the symmetric group $S_n$.
\begin{thm} \label{thm:qNoether_main_result}
Define a set of elements $e_k^{(i)}\in \K_q(\bar x,\bar y)^{S_n}$
for $i\in\iv{0}{n}$, $k\in\iv{0}{n-i}$ recursively by
\begin{subequations}\label{eq:eki_def}
\begin{align}
e_k^{(0)} =& e_k,\qquad \forall k\in\iv{0}{n}, \\
e_k^{(1)} =& t_{k+1},\qquad \forall k\in\iv{0}{n-1},\\
e_k^{(i)} =& e_{k+1}^{(i-2)}e_0^{(i-1)}e_{n-i+1}^{(i-1)}
-(-1)^{k+1} e_0^{(i-2)}  e_{n-i-k}^{(i-1)}  e_0^{(i-1)}  \\
&-(-1)^{n-i+1-k} e_{n-i+2}^{(i-2)} e_{n-i+1-k}^{(i-1)}e_{n-i+1}^{(i-1)},
\qquad\forall k\in\iv{0}{n-i}, \forall i\in\iv{2}{n}, \nonumber
\end{align}
\end{subequations}
where $e_d$ and $t_j$ were defined in \eqref{eq:ed_def} and \eqref{eq:PX} respectively.
Let
\begin{subequations}\label{eq:qNoether_XY_def}
\begin{align}
(X_1,X_2,\ldots,X_n)= & (e_n^{(0)}, e_{n-1}^{(1)},\ldots ,e_1^{(n-1)}),\\
(Y_1,Y_2,\ldots,Y_n)= & (e_0^{(1)}, e_0^{(2)},\ldots e_0^{(n)} ),
\end{align}
\end{subequations}
and put
\begin{subequations}\label{eq:qNoether_hatXhatY_def}
\begin{gather}
\widehat{X}_1  = X_1,\quad \widehat{X}_i = Y_{i-1}^{(-1)^i} X_i^{(-1)^{i+1}},\qquad\forall i\in\iv{2}{n},\\
\widehat{Y}_1  = Y_1,\quad \widehat{Y}_2 = Y_1^{-2}Y_2,\quad\widehat{Y}_j=Y_{j-2}^{-1}Y_{j-1}^{-2}Y_j,
\qquad\forall j\in\iv{3}{n},
\end{gather}
\end{subequations}
Then there is an isomorphism of $\K$-algebras
\begin{align}
\label{eq:qNoether_isomorphism}
\K_q(\bar x,\bar y)\overset{\sim}{\longrightarrow} & \K_q(\bar x,\bar y)^{S_n}
\intertext{given by}
\label{eq:qNoether_isomorphism_2}
x_k\longmapsto &\widehat{X}_k,\qquad\forall k\in\iv{1}{n},\\
\label{eq:qNoether_isomorphism_3}
y_k\longmapsto &\widehat{Y}_k,\qquad\forall k\in\iv{1}{n}.
\end{align}
\end{thm}
\begin{proof}
First we prove that for each $i\in\iv{1}{n}$, the elements
\begin{subequations}\label{eq:ET_phone_home}
\begin{align}
(E_0,\ldots,E_{n-i+1})= &(e_0^{(i-1)},\ldots, e_{n-i+1}^{(i-1)}),\\
(T_1,\ldots,T_{n-i+1})= &(e_0^{(i)},\ldots, e_{n-i}^{(i)}),
\end{align}
\end{subequations}
satisfy relations \eqref{eq:TiTj_assumption},\eqref{eq:EkEl_assumption} and \eqref{eq:TjEk_assumption}
 with $n$ replaced by $n-i+1$, and
\begin{equation}\label{eq:F_semi-opposite_def}
\F=\F_i:=\begin{cases}\K_q(\bar x,\bar y)^{S_n}& \text{if $i$ is odd,}\\
 \big(\K_q(\bar x,\bar y)^{S_n}\big)^{\mathrm{op}}& \text{if $i$ is even.}\end{cases}
\end{equation}
We prove this by induction on $i$.
For $i=1$ this follows from Proposition~\ref{prp:tj_ek}.
For $i>1$ we may, by the induction hypothesis, apply Proposition~\ref{prp:inductive_procedure}
with $n$ replaced by $n-i+2$ and
\begin{subequations}\label{eq:ind_step_temp}
\begin{align}
(E_0,\ldots,E_{n-i+2}) &= (e_0^{(i-2)},\ldots, e_{n-i+2}^{(i-2)}), \\
(T_1,\ldots,T_{n-i+2}) &= (e_0^{(i-1)},\ldots, e_{n-i+1}^{(i-1)}),
\end{align}
\end{subequations}
and $\F=\F_{i-1}$. Substituting \eqref{eq:ind_step_temp} into
 \eqref{eq:tilde_T_def}, \eqref{eq:tilde_E_def},
we obtain
\begin{equation}\label{eq:ET_ph2}
(e_0^{(i-1)},\ldots,e_{n-i+1}^{(i-1)})= (\widetilde{E}_0,\ldots,\widetilde{E}_{n-i+1}),
\end{equation}
and, \emph{in the algebra $\F_{i-1}$},
\[
(e_0^{(i)},\ldots,e_{n-i}^{(i)})= (\widetilde{T}_1,\ldots,\widetilde{T}_{n-i+1}),
\]
by the definition of $e_k^{(i)}$. Thus, keeping in mind \eqref{eq:tilde_T_alt},
we obtain that in $\K_q(\bar x,\bar y)^{S_n}$,
\begin{equation}\label{eq:ET_ph3}
(e_0^{(i)},\ldots,e_{n-i}^{(i)})= q^{\delta_{i-1\in 2\Z}} (\widetilde{T}_1,\ldots,\widetilde{T}_{n-i+1}).
\end{equation}
However, the possible extra factor $q$ does matter;
the conclusion from Proposition \ref{prp:inductive_procedure}
that relations \eqref{eq:tilde_T_commute},\eqref{eq:tilde_E_commute}, and \eqref{eq:tilde_T_tilde_E} (with
$n$ replaced by $n-i+2$) hold
in $\F_{i-1}$ implies that,
choosing $E_k, T_j$ as in \eqref{eq:ET_phone_home}, relations
 \eqref{eq:TiTj_assumption},\eqref{eq:EkEl_assumption},\eqref{eq:TjEk_assumption}
(with $n$ replaced by $n-i+1$) hold in the algebra $\F_i$. This proves the induction step.

In particular, by \eqref{eq:TiTj_assumption} and \eqref{eq:EkEl_assumption},
\begin{subequations}\label{eq:e_identities_a}
\begin{align}
[e_j^{(i)},e_0^{(i)}] =&0,\quad\forall j\in\iv{0}{n-i},\forall i\in\iv{0}{n},\\
[e_j^{(i)},e_{n-i}^{(i)} ]=& 0, \quad\forall j\in\iv{0}{n-i},\forall i\in\iv{0}{n}.
\end{align}
By \eqref{eq:TjE0} and \eqref{eq:TjEn} we have, in $\K(\bar x,\bar y)^{S_n}$,
\end{subequations}
\begin{subequations}\label{eq:e_identities_b}
\begin{align}
\label{eq:e_identities_1}
[e_{j}^{(i+1)}, e_0^{(i)}]=&0,\quad \forall j\in\iv{0}{n-i-1},\forall i\in\iv{0}{n-1},\\
\label{eq:e_identities_2}
[e_{j}^{(i+1)}, e_{n-i}^{(i)}]_{q^{(-1)^i}} =&0,\quad \forall
j\in\iv{1}{n-i-1},\forall i\in\iv{0}{n-1}.
\end{align}
\end{subequations}
More generally, the following relations hold:
\begin{subequations}\label{eq:e_things}
\begin{align}
 \label{eq:e_thing_1}
[e_j^{(k)}, e_0^{(i)}] &=0,\quad \forall j\in\iv{0}{n-k}, 0\le i\le k\le n, \\
\label{eq:e_thing_2}
[ e_j^{(k)}, e_{n-i}^{(i)} ]_{q^{(-1)^i\cdot a_{k-i}}} &= 0,\quad \forall j\in\iv{0}{n-k}, 0\le i\le k\le n,
\end{align}
\end{subequations}
where $a_k\in\Z$ is defined by the recursion relation
\begin{equation}\label{eq:a_recursion}
a_k=2a_{k-1}+a_{k-2},\;\; a_0=0,\; a_1=1.
\end{equation}
To prove this we use induction on $k-i$. For $k-i=0$ and $k-i=1$,
relations \eqref{eq:e_things} follow from \eqref{eq:e_identities_a} and \eqref{eq:e_identities_b} respectively.
Assume $k-i>1$. By the induction hypothesis we have, for any $j_1,j_2,j_3$,
\begin{equation}\label{eq:e_stuff1}
[ e_{j_1}^{(k-2)}e_{j_2}^{(k-1)}e_{j_3}^{(k-1)} , e_0^{(i)} ] = 0
\end{equation}
and
\begin{align}
 e_{j_1}^{(k-2)}e_{j_2}^{(k-1)}e_{j_3}^{(k-1)}\cdot e_{n-i}^{(i)}
 = & q^{(-1)^i\cdot (a_{k-2-i}+2a_{k-1-i}) } e_{n-i}^{(i)}
\cdot
 e_{j_1}^{(k-2)}e_{j_2}^{(k-1)}e_{j_3}^{(k-1)} \nonumber \\
= & q^{ (-1)^i\cdot a_{k-i}} e_{n-i}^{(i)} \cdot e_{j_1}^{(k-2)}e_{j_2}^{(k-1)}e_{j_3}^{(k-1)}
\label{eq:e_stuff2}
\end{align}
Using \eqref{eq:e_stuff1},\eqref{eq:e_stuff2} and the definition, \eqref{eq:eki_def},
of $e_j^{(k)}$, we obtain \eqref{eq:e_things}.

Using $X_i,Y_j$ given in \eqref{eq:qNoether_XY_def}, relations \eqref{eq:e_things} imply that
\begin{subequations}\label{eq:qNoether_XY_relations}
\begin{align}
[Y_k, Y_i]=& 0,\quad\forall k,i\in\iv{1}{n},\\
[X_k, X_i]_{q^{(-1)^{i+1}\cdot a_{k-i}}} =& 0,\quad k\ge i,\\
[Y_k, X_i]=& 0,\quad k<i,\\
[Y_k, X_i]_{q^{(-1)^{i+1}\cdot a_{k-i+1}}} =&0,\quad k\ge i.
\end{align}
\end{subequations}
This means that, putting $(Z_1,\ldots,Z_{2n})=(X_1,Y_1,X_2,Y_2,\ldots,X_n,Y_n)$,
we have
\begin{equation}
Z_iZ_j=q^{s_{ij}}Z_jZ_i,
\end{equation}
where $S=(s_{ij})$ is the $2n\times 2n$ skew-symmetric
integer matrix
\begin{equation}
S=\begin{bmatrix}
0      & -1   & -1    & -2   & -2  & -5  & \cdots & -a_{n-1} & -a_n    \\
1      &  0   &  0    & 0    & 0   &  0  & \cdots &  0       &  0      \\
1      &  0   &  0    & 1    & 1   &  2  & \cdots &  a_{n-2} & a_{n-1} \\
2      &  0   & -1    & 0    & 0   &  0  & \cdots &  0       &  0      \\
2      &  0   & -1    & 0    & 0   & -1  & \cdots & -a_{n-3} &-a_{n-2} \\
5      &  0   & -2    & 0    & 1   &  0  & \cdots &  0       &  0      \\
\vdots &\vdots&\vdots &\vdots&\vdots&\vdots&\ddots&\vdots    &\vdots   \\
a_{n-1}& 0    &-a_{n-2}& 0   &a_{n-3}& 0 & \cdots &  0       & (-1)^n  \\
a_n    & 0    &-a_{n-1}& 0   &a_{n-2}& 0 & \cdots &(-1)^{n+1}& 0
\end{bmatrix}
\end{equation}
The matrix $S$ may be brought to normal form as follows. Take
\setcounter{MaxMatrixCols}{20}
\[
U=\begin{bmatrix}
1     &      &      &      &      &      &      &          &      &          &      &                  \\
      & 1    & 1    &-2    &      & -1   &      &          &      &          &      &                  \\
      &      &-1    &      &      &      &      &          &      &          &      &                  \\
      &      &      & 1    & -1   & -2   &      &   -1     &      &          &      &                  \\
      &      &      &      &  1   &      &      &          &      &          &      &                  \\
      &      &      &      &      &  1   &  1   &   -2     &      &   -1     &      &                  \\
\vdots&\vdots&\vdots&\vdots&\vdots&\vdots&\ddots&\ddots    &\ddots&\ddots    &\ddots&            \\
      &      &      &      &      &      &\cdots&            1   &(-1)^{n-1}& -2   &          & -1     \\
      &      &      &      &      &      &\cdots&                &(-1)^n    &      &          &        \\
      &      &      &      &      &      &\cdots&                &          &  1   &(-1)^n    & -2     \\
      &      &      &      &      &      &\cdots&                &          &      &(-1)^{n+1}&        \\
      &      &      &      &      &      &\cdots&                &          &      &          &  1
\end{bmatrix}
\]
where zero entries were omitted. Then $U^tSU$ is block diagonal
with $n$ copies of
$\left[\begin{smallmatrix}0&-1\\1&0\end{smallmatrix}\right]$ on
the diagonal. As in the proof of Proposition \ref{prp:normalform},
we see that the matrix $U$ corresponds exactly to the change of
variables \eqref{eq:qNoether_hatXhatY_def}. This proves that
\begin{equation}\label{eq:hatXhatY_relations}
\begin{gathered}
\widehat{Y}_i\widehat{X}_j= q^{\delta_{ij}} \widehat{X}_j\widehat{Y}_i, \quad\forall i,j\in\iv{1}{n},\\
[\widehat{X}_i,\widehat{X}_j]=[\widehat{Y}_i,\widehat{Y}_j]=0,\quad\forall i,j\in\iv{1}{n}.
\end{gathered}
\end{equation}
Moreover, one can easily check that the set
$\{\widehat{X}_1,\ldots,\widehat{X}_n,\widehat{Y}_1,\ldots,\widehat{Y}_n\}$
generate the skew subfield isomorphic to the skew subfield
generated by $X_1, \ldots, X_n, Y_1, \ldots, Y_n$.
 Alternatively, one may prove
\eqref{eq:hatXhatY_relations} directly by using
\eqref{eq:qNoether_XY_relations} and
\eqref{eq:qNoether_hatXhatY_def}.

Now \eqref{eq:hatXhatY_relations} implies the existence of a
unique $\K$-algebra homomorphism \eqref{eq:qNoether_isomorphism}
satisfying
\eqref{eq:qNoether_isomorphism_2},\eqref{eq:qNoether_isomorphism_3}.
Since the domain is a skew field, it is sufficient to show that
the homomorphism  is surjective. It follows from
Proposition~\ref{prp:inductive_procedure} that the set
$\{X_1,\ldots,X_n,Y_1,\ldots,Y_n\}$ generates $\K_q(\bar x,\bar y)^{S_n}$ as a skew field over $\K$. Hence
\[ \{\widehat{X}_1,\ldots,\widehat{X}_n,\widehat{Y}_1,\ldots,\widehat{Y}_n\} \]
also generates $\K_q(\bar x,\bar y)^{S_n}$ and thus the
homomorphism \eqref{eq:qNoether_isomorphism} is surjective. This
concludes the proof.
\end{proof}

\begin{cor}\label{cor-weyl}
We have an isomorphism of $\K$-algebras
\begin{equation}
\big(\Frac (A_1^q(\K)^{\otimes_\K n})\big)^{S_n}\simeq
\Frac (A_1^q(\K)^{\otimes_\K n}).
\end{equation}
\end{cor}
\begin{proof}
Follows directly from Theorem~\ref{thm:qNoether_main_result}
and Proposition~\ref{prp:qWeyl_iso}, noting that the isomorphism
in Proposition~\ref{prp:qWeyl_iso} commutes with the $S_n$-action.
\end{proof}

We will need one more property of the isomorphism \eqref{eq:qNoether_isomorphism}.
For $r\in\K\setminus\{0\}$ we define two automorphisms $\al_r,\be_r$ of $\K_q(\bar x,\bar y)$ as follows:
\begin{gather}
\al_r, \be_r:\K_q(\bar x,\bar y)\to \K_q(\bar x,\bar y),\\
\al_r(x_j)=x_j, \quad
\al_r(y_j)=r\cdot y_j,\quad\forall j\in\iv{1}{n}\\
\be_r(x_j)=x_j, \quad
\be_r(y_j)=r^{\delta_{1j}}\cdot y_j,\quad\forall j\in\iv{1}{n}.
\end{gather}
Similarly to how one proves the commutation relations
\[
[\widehat{X}_1,\widehat{X}_j]=0,\quad \widehat{X}_1\widehat{Y}_j\widehat{X}_1^{-1}=q^{-\delta_{1j}}\widehat{Y}_j,\quad\forall j\in\iv{1}{n}\]
one can verify the following result.
\begin{lem} \label{lem:alpha_beta}
The isomorphism $g:\K_q(\bar x,\bar y)^{S_n}\to \K_q(\bar x,\bar y)$ constructed in
Theorem \ref{thm:qNoether_main_result} satisfies
\begin{equation}\label{eq:alpha_beta}
g\circ \al_r \circ g^{-1} = \be_r
\end{equation}
for all $r\in \K\setminus\{0\}$.
\end{lem}


\section{The \texorpdfstring{$q$}{q}-difference Noether problem for classical Weyl
groups}\label{section-weyl} Let $W(B_n)=W(C_n)=S_n\ltimes
(\Z/2\Z)^n$ be the Weyl group of type $B_n$ (equivalently, of type
$C_n$). The group $W(B_n)$ acts naturally on $\K_q(\bar x,\bar
y)\simeq \mathcal{O}_q(\Bbbk^{2n})$ by
\begin{subequations}\label{eq:Wm_action}
\begin{gather}
\label{eq:Sm_action}
\zeta(x_i)=x_{\zeta(i)},\quad \zeta(y_i)=y_{\zeta(i)},
\quad \forall\zeta\in S_n,\;\forall i\in\iv{1}{n},\\
\label{eq:Em_action}
\alpha(x_i)=(-1)^{\alpha_i}x_i,\quad \alpha(y_i)=(-1)^{\alpha_i}y_i,
\quad \forall\alpha\in(\Z/2\Z)^n,\;\forall i\in\iv{1}{n}.
\end{gather}
\end{subequations}
Let $\mathcal{E}_n=\{\al=(\al_1,\ldots,\al_n)\in (\Z/2\Z)^n\mid \al_1+\cdots+\al_n=0\}$
 and $W(D_n)=S_n\ltimes\mathcal{E}_n$ be the
Weyl group of type $D_n$.


In Theorem \ref{thm:qNoether_main_result} we solved the q-difference Noether problem for the Weyl group of type $A_n$.
In this section we will show that the other cases ($B_n, C_n, D_n$) can be reduced to that case.
First note that by replacing $y_i$ by $x_iy_i$ in $\K_q(\bar x,\bar y)$ we can, and will, assume that $(\Z/2\Z)^n$ fixes $y_i$ for all $i$, so that \eqref{eq:Em_action} is replaced by
\begin{equation}
\alpha(x_i)=(-1)^{\alpha_i}x_i,\quad \alpha(y_i)=y_i,
\quad \forall\alpha\in(\Z/2\Z)^n,\;\forall i\in\iv{1}{n}.
\end{equation}

We start with the case $B_n$, which is the easiest.
\begin{thm}\label{thm:qNoether_B}
The $q$-difference Noether problem for the Weyl group of type $B_n$ as a positive solution.
More precisely, there exist $\K$-algebra isomorphisms
\begin{equation}\label{eq:qNoether_B}
\K_q(\bar x,\bar y)^{W(B_n)} \simeq \K_{q^2}(\bar x,\bar y)^{S_n}\simeq \K_{q^2}(\bar x,\bar y).
\end{equation}
\end{thm}
\begin{proof}
Using that
\[\big\{x_1^{k_1}\cdots x_n^{k_n}\cdot y_1^{k_{n+1}}\cdots y_n^{k_{2n}}\mid k\in\Z^{2n}\big\}\]
is a $\K$-basis for $\K_q[\bar x,\bar y]$ it is easy to see that there is an isomorphism
of $\K$-algebras
\begin{align*}
 \K_{q^2}[\bar x,\bar y]\overset{\sim}{\longrightarrow} & \K_q[\bar x,\bar y]^{(\Z/2\Z)^n}
\intertext{given by}
 x_i \longmapsto & x_i^2, \\
 y_i \longmapsto & y_i.
\end{align*}
Taking skew field of fractions on both sides, followed by taking $S_n$-invariants we obtain that
\[\K_{q^2}(\bar x,\bar y)^{S_n}
\simeq
\Big(\K_q(\bar x,\bar y)^{(\Z/2\Z)^n}\Big)^{S_n}=
\K_q(\bar x,\bar y)^{W(B_n)},
 \]
 which together with \eqref{eq:qNoether_isomorphism} proves \eqref{eq:qNoether_B}.
\end{proof}

For the remaining type $D_n$ case, we need the following lemma.
\begin{lem} \label{lem:qNoether_D_lem1}
The algebra $\K_q(\bar x,\bar y)^{W(D_n)}$ is free as a left $\K_q(\bar x,\bar y)^{W(B_n)}$-module
with basis $\{1, x_1x_2\cdots x_n\}$.
\end{lem}
\begin{proof}
We must prove that
\begin{equation}\label{eq:qNoether_D_lem1}
\K_q(\bar x,\bar y)^{W(B_n)} \oplus \K_q(\bar x,\bar y)^{W(B_n)}\cdot x_1x_2\cdots x_n =
\K_q(\bar x,\bar y)^{W(D_n)}
\end{equation}
Let $\ga\in W(B_n)$ be a representative for the nontrivial element in
$W(B_n)/W(D_n)\simeq \Z/2\Z$. For example we may take $\ga=(1,0,\ldots,0)\in (\Z/2\Z)^n\subseteq W(B_n)$.
Then $\ga$ acts as an order two $\K$-algebra automorphism of
$\K_q(\bar x,\bar y)^{W(D_n)}$. By polarization, we get a decomposition of
$\K_q(\bar x,\bar y)^{W(D_n)}$ into $\pm 1$ eigenspaces. The $+1$ eigenspace of $\ga$ is obviously equal to
$\K_q(\bar x,\bar y)^{W(B_n)}$. Since $x_1x_2\cdots x_n$ belongs to the $-1$ eigenspace and is invertible,
it is easy to see that the $-1$ eigenspace of $\ga$ equals
\[ \K_q(\bar x,\bar y)^{W(B_n)}\cdot x_1x_2\cdots x_n.\]
 This proves \eqref{eq:qNoether_D_lem1}.
\end{proof}

We are now ready to prove the following.
\begin{thm} \label{thm:qNoether_D}
 The $q$-difference Noether problem for the Weyl group $W_n=W(D_n)$ of type $D_n$ has a positive solution.
Explicitly, there exists a $\K$-algebra isomorphism
\begin{equation}
\K_q(\bar x,\bar y)^{W_n}\simeq \Frac\big( \K_q[x,y]\otimes_\K \K_{q^2}[x,y]^{\otimes_\K (n-1)}\big).
\end{equation}
\end{thm}
\begin{proof}

The isomorphism $g=g_2\circ g_1$ where
$g_1:\K_q(\bar x,\bar y)^{W(B_n)}\overset{\sim}{\to} \K_{q^2}(\bar x,\bar y)^{S_n}$
and $g_2:\K_{q^2}(\bar x,\bar y)^{S_n}\overset{\sim}{\to} \K_{q^2}(\bar x,\bar y)$,
obtained in the proof of Theorem~\ref{thm:qNoether_B}, satisfies
$g(x_1^2x_2^2\cdots x_n^2)=x_1$. We also have a $\K$-algebra monomorphism
\begin{align*}
k:\K_{q^2}(\bar x,\bar y)\hookrightarrow &\Frac\big(\K_q[x_1,y_1]\otimes_\K \K_{q^2}[x_2,y_2]
\otimes_\K\cdots \otimes_\K \K_{q^2}[x_n,y_n]\big),\\
x_1\mapsto & x_1^2,\\
x_i\mapsto & x_i,\quad\forall i\in\iv{2}{n},\\
y_i\mapsto & y_i,\quad\forall i\in\iv{1}{n}.
\end{align*}
Similarly to Lemma \ref{lem:qNoether_D_lem1} we have a direct sum decomposition
\begin{equation}\label{eq:qNoether_D_eq1}
\Frac\big(\K_q[x_1,y_1]\otimes_\K \K_{q^2}[x_2,y_2]
\otimes_\K\cdots \otimes_\K \K_{q^2}[x_n,y_n]\big) = \im k \oplus (\im k)\cdot x_1.
\end{equation}
Using Lemma \ref{lem:qNoether_D_lem1}, we now define
\begin{equation}
f:\K_q(\bar x,\bar y)^{W(D_n)}\longrightarrow  \Frac\big(\K_q[x_1,y_1]\otimes_\K \K_{q^2}[x_2,y_2]
\otimes_\K\cdots \otimes_\K \K_{q^2}[x_n,y_n]\big)
\end{equation}
by
\begin{equation}
f(a+b\cdot x_1x_2\cdots x_n) = (k\circ g)(a)+(k\circ g)(b)\cdot x_1,\quad\forall a,b\in\K_q(\bar x,\bar y)^{W(B_n)}.
\end{equation}
By \eqref{eq:qNoether_D_eq1}, $f$ is a surjective map. Furthermore, the restriction of $f$
to $\K_q(\bar x,\bar y)^{W(B_n)}$ is a homomorphism and
 $x_1^2=f\big((x_1x_2\cdots x_n)^2\big)$.
Thus, to prove that $f$ is
a homomorphism it is thus enough to show that
\begin{equation}\label{eq:qNoether_D_eq2}
(k\circ g)\big(x_1x_2\cdots x_n \cdot a\cdot (x_1x_2\cdots x_n)^{-1}\big) = x_1\cdot (k\circ g)(a)\cdot x_1^{-1},\quad\forall a\in\K_q(\bar x,\bar y)^{W(B_n)}.
\end{equation}
Recall the automorphisms $\al_r,\be_r$ from Lemma \ref{lem:alpha_beta}. We have
\begin{equation}\label{eq:qNoether_D_xkgx}
 x_1\cdot (k\circ g)(a)\cdot x_1^{-1} = (k\circ \beta_{q^{-1}} \circ g_2\circ g_1)(a).
\end{equation}
By Lemma \ref{lem:alpha_beta}, $\beta_{q^{-1}} = g_2 \circ \al_{q^{-1}} \circ g_2^{-1}$. So
\eqref{eq:qNoether_D_xkgx} equals
\[ (k\circ g_2 \circ \al_{q^{-1}} \circ g_1)(a)=(k\circ g)\big(x_1x_2\cdots x_n\cdot a\cdot (x_1x_2\cdots x_n)^{-1}\big)\]
which proves \eqref{eq:qNoether_D_eq2}.
This proves that $f$ is a surjective $\K$-algebra homomorphism. It is injective since its domain
is a skew-field.
\end{proof}

Theorem~\ref{thm:qNoether_B} and Theorem~\ref{thm:qNoether_D} complete the proof of Theorem~II.

\begin{rem}
We note that a positive solution to the q-difference Noether problem for classical Weyl groups
in the case $q=1$ can be deduced from \cite[Remark~3]{Mi}.
\end{rem}
\section{Reduction via Galois rings}

For the rest of the paper we specialize to $\K=\C$ as ground field,
and assume that $q\in\C\setminus\{0\}$ is not a root of unity.

We use the theory of Galois rings
\cite{FO} to reduce the quantum Gelfand-Kirillov
conjecture to the $q$-difference Noether problem.

\subsection{Galois rings}\label{sec:GaloisRings}

In this subsection, $\Gamma$ denotes an integral domain, $K$ the field of fractions of $\Gamma$,
$K\subseteq L$ a finite Galois extension with Galois group $G=\Gal(L/K)$,
and $\Mscr$ a monoid  acting on $L$ by automorphisms. We will assume that $\Mscr$ is
 $K$-separating, that is $m_1|_K=m_2|_K$ implies $ m_1=m_2$
for $m_1,m_2\in\Mscr$. The group $G$ acts naturally
on $\Mscr$ by conjugations and thus on the skew monoid ring $L\ast \Mscr$ by automorphisms.
We denote the $G$-invariants in $L\ast\Mscr$ by $(L\ast\Mscr)^G$.

If $u=\sum_{m\in\Mscr} a_mm \in L\ast \Mscr$, we put
$\Supp(u)=\{m\in\Mscr\mid a_m\neq 0\}$. For $\varphi\in\Mscr$, let
$\Stab_G(\varphi)$ be the stabilizer subgroup of $G$ at $\varphi$
and $T_\varphi\subseteq G$ be a set of representatives for
$G/\Stab_G(\varphi)$ (the set of orbits of the action of
$\Stab_G(\varphi)$ on $G$ by conjugations). For $a\in L$, put
\begin{equation}\label{eq:GinvElt}
[a\varphi]:=\sum_{g\in T_\varphi} a^g\varphi^g.
\end{equation}
Then $[a \varphi]\in (L\ast \Mscr)^G$, \cite[Lemma~2.1]{FO}.

\begin{dfn}[\cite{FO}, Definition 3]
A finitely generated $\Gamma$-subring $U\subseteq (L\ast\Mscr)^G$
 is called a \emph{Galois $\Gamma$-ring} if $UK=KU=(L\ast\Mscr)^G$.
\end{dfn}

\begin{prp}[\cite{FO}, Proposition 4.1]
\label{prp:GaloisRing}
Suppose $U$ is a $\Gamma$-subring of $(L\ast\Mscr)^G$ generated by
$u_1,\ldots,u_k\in U$. If $\cup_{i=1}^k \Supp(u_i)$ generate $\Mscr$ as a monoid,
then $U$ is a Galois $\Gamma$-ring in $(L\ast\Mscr)^G$.
\end{prp}

\begin{proof}
Since the proof of \cite[Proposition~4.1]{FO} is rather sketchy we
provide the details for convenience. Consider a $K$-subbimodule
$V=Ku_{1}K+$ $\dots+$ $Ku_{k}K$ in $(L\ast\Mscr)^G$. It follows
from the proof of \cite[Lemma~4.1]{FO} that for any $i$ and any
$m\in \Supp(u_i)$ there exists $a\in L$ such that $[am]\in Ku_iK$.
Thus the bimodule $V$ contains the elements $[a_{1}\vi_{1}],$
$\dots,$ $[a_{t}\vi_{t}]$, where $\vi_{1}^{g},$ $\dots,$
$\vi_{t}^{g},$ $g\in G,$ generate $\Mscr$. Now consider a
subalgebra $U'\subset U$ generated over $\Gamma$ by
$[a_{i}\vi_{i}],$ $i=1, \ldots, t$. Since
$$\Supp([am]\Gamma[a'm'])=\Supp[am]\Supp[a'm'],$$
then given $\vi\in \Mscr$ one can find $a\in L$ such that
$[a\vi]\in U'$. Moreover, $a\in L^{\Stab_G(\varphi)}$. Now we use
the fact that $K\vi(\Gamma)=\vi(K)$ and hence
$$K(\Gamma[a\vi]\Gamma)=[K\Gamma\vi(\Gamma)a\vi]=[K\vi(K)a\vi].$$
Thus  $K(\Gamma[a\vi]\Gamma)=[L^{\Stab_G(\varphi)}\vi]$ and
$KU\simeq (L\ast\Mscr)^G$. Similarly, $UK\simeq (L\ast\Mscr)^G$.
We conclude that $U$ is a Galois $\Gamma$-ring in
$(L\ast\Mscr)^G$.

\end{proof}

\subsection{The center of \texorpdfstring{$U_q(\mathfrak{gl}_N)$}{Uq(glN)}}
\label{sec:center}

It is known that the center $Z_N$ of $U_N=U_q(\mathfrak{gl}_N)$
is generated by the quantum Casimir operators
constructed by Bracken, Gould and Zhang \cite{BGZ} and by the
element $(K_1\ldots K_N)^{-1}$ \cite{Li}. Here we recall
some facts that will be used in later sections.

Let $U_N^0$, (respectively $U_N^\pm$) be the subalgebra of $U_N$
generated by $K_i, K_i^{-1}, i\in\iv{1}{N}$ (respectively
$E_j^\pm, j\in\iv{1}{N-1}$). By the quantum PBW theorem we have
$U_N=U_N^+U_N^0U_N^-$. Thus each $a\in U_N$ can be uniquely
decomposed as $a=a^{(0)}+a'$, where $a^{(0)}\in U_N^0$ and
 $a'\in \sum_j E_j^+U_N + U_NE_j^-$.
The \emph{quantum Harish-Chandra homomorphism} $h_N:Z_N\to U_N^0$ is
defined by $h_N(z)=z^{(0)}$.

Put $\widetilde{K}_i=q^{-i}K_i$. We may regard $U_N^0$ as a
Laurent polynomial algebra in the variables $\widetilde{K}_i$.
Let $W_N=S_N\ltimes\mathcal{E}_N$, the Weyl group of type $D_N$,
act on $U_N^0$ by permutations and sign changes of $\widetilde{K}_i,\; i\in\iv{1}{N}$.
The following lemma give a description of the center of $U_N$.
\begin{lem}\label{lem:image_of_HC}
We have $\C$-algebra isomorphisms 
\begin{equation}\label{eq:image_of_HC}
Z_N\overset{h_N}{\simeq} (U_N^0)^{W_N}\simeq \C[z_1,\ldots,z_{N-1}][z_{N}^{\pm 1}].
\end{equation}
\end{lem}
\begin{proof}
Let $(U_N^0)_{\mathrm{ev}}$ denote the subalgebra of $U_N^0$
generated by $K_i^{\pm 2}$, $i\in\iv{1}{N}$. By \cite[Lemma 2.1]{Li}, $h_N$ is injective and its image is generated by
$\big((U_N^0)_{\mathrm{ev}}\big)^{S_N}$ and the element
$I_N^{-1}$, where $I_N:=K_1K_2\cdots K_N$. Note that $\mathcal{E}_N$ fixes $K_i^{\pm 2}$ for all $i\in\iv{1}{N}$ and also fixes $I_N^{-1}$ since there are only an even number of sign changes.
Thus the image of $h_N$ is contained in $(U_N^0)^{W_N}$. For the converse inclusion, one can check that the order two $\C$-algebra automorphism of $U_N^0$ given by $K_j\mapsto (-1)^{\delta_{1j}}K_j$ for $j\in\iv{1}{N}$ preserves the subalgebra $(U_N^0)^{W_N}$. The $+1$ eigenspace of $(U_N^0)^{W_N}$ coincides with $\big((U_N^0)_{\mathrm{ev}}\big)^{S_N}$.
The element $I_N$ belongs to the $-1$ eigenspace of $(U_N^0)^{W_N}$. Multiplying any element of the $-1$ eigenspace by $I_N$ we get an element of the $+1$ eigenspace. Since $I_N$ is invertible, it follows that the $-1$ eigenspace of $(U_N^0)^{W_N}$ is equal to $I_N^{-1}\cdot (U_N^0)_{\mathrm{ev}}$. This proves that the image of $h_N$ equals $(U_N^0)^{W_N}$.

For the second map in \eqref{eq:image_of_HC} we define
\begin{align*}
f:
\C[z_1,\ldots,z_{N-1}][z_N^{\pm 1}]
\longrightarrow & (U_N^0)^{W_N},\\
 z_d \longmapsto & e_d(\widetilde{K}_1^2,\ldots,\widetilde{K}_N^2) ,\quad\forall d\in\iv{1}{N-1},\\
z_N\longmapsto & \widetilde{K}_1\widetilde{K}_2\cdots\widetilde{K}_N,
\end{align*}
where $e_d$ is the elementary symmetric polynomial in $N$ variables of degree $d$.
Since $I_N$ is invertible and $U_N^0$ is commutative, $f$ is a  well-defined $\C$-algebra homomorphism. By the previous paragraph, any element of $(U_N^0)^{W_N}$ can be written as a sum of elements of the form $I_N^{-k}\cdot u$, where $k\in\Z_{\ge 0}$ and $u$ is a symmetric polynomial in $\widetilde{K}_i^2$, $i\in\iv{1}{N}$.
By Newton's theorem and that $f(z_N^2)=e_N(\widetilde{K}_1^2,\ldots,\widetilde{K}_N^2)$, we conclude that $u$, hence $I_N^{-k}\cdot u$ lies in the image of $f$. This proves that $f$ is surjective. To prove that $f$ is injective, it is enough to prove that $f(z_1),\ldots,f(z_N)$ are algebraically independent over $\C$. By applying the involution $K_j\mapsto (-1)^{\delta_{1j}}K_j$ from the previous paragraph, it is enough to prove that $f(z_1),\ldots,f(z_{N-1}), f(z_N)^2$ are algebraically independent, which follows from Newton's theorem.
\end{proof}

\begin{rem}
We note that \cite[Eq. (2.5)]{Li} can be regarded as a
$q$-deformation of a formula of Zhelobenko \cite{Zh}.
\end{rem}

\subsection{Gelfand-Tsetlin modules over \texorpdfstring{$U_q(\mathfrak{gl}_N)$}{Uq(glN)}}

Gelfand-Tsetlin bases for finite-dimensional irreducible
representations of $U_q(\mathfrak{gl}_N)$ were obtained in
\cite{UTS}. Similarly to the classical $U(\mf{gl}_N)$-case, the
bases consist of finite sets of \emph{tableaux}, i.e.
double-indexed families $(\la_{mi})_{1\le i\le m\le N}$ of
integers, satisfying certain conditions. The action of the
generators $E_i^\pm$ and $K_j$ on these tableaux are given by
$q$-analogues of the classical Gelfand-Tsetlin formulas.

Mazorchuk and Turowska \cite{MT} used these formulas to define a
family of $U_q(\mathfrak{gl}_N)$-modules (in fact they used
the algebra obtained from $U_{q^2}(\mathfrak{gl}_N)$
by adjoining $K_j^{\pm 1/2}$, but the results are the same), the so called \emph{generic Gelfand-Tsetlin
modules}, which are always infinite-dimensional and not
necessarily simple. The bases are now parametrized by tableaux
with complex entries $\la=(\la_{mi})_{1\le i\le m\le
N}\in\C^{N(N+1)/2}$. The only restriction on the tableaux is that
they should be \emph{admissible}. By definition, a tableau $\la$
is admissible if $q^{2(k+\la_{mi}-\la_{mj})}\neq 1$ for all
$k\in\Z$ and all $1\le i,j\le m\le N$.




The following theorem gives their construction.
For $x\in\C$ we put
\[[x]_q:=\frac{q^x-q^{-x}}{q-q^{-1}}.\]

\begin{thm}[\cite{MT}, Theorem 2]
\label{MTthm1}
 To each admissible tableau $\la$ there exists a
 $U_q(\mathfrak{gl}_N)$-module $V(\la)$ with basis
 $B(\la)=\big\{[\la+\ga]\mid \ga \in \Z^{N(N-1)/2}]\big\}$
 and action given by
\begin{equation}\begin{aligned}
E_m^\pm [\mu] &= \sum_{i=1}^m a_{mi}^\pm(\mu)[\mu\pm\delta^{mi}],
\qquad m=1,\ldots, N-1,\\
K_m[\mu]&=q^{\sum_{i=1}^m \mu_{mi} -
\sum_{i=1}^{m-1}\mu_{m-1,i}}[\mu], \qquad m=1,\ldots, N,
\end{aligned}\end{equation}
for any $\mu\in B(\la)$, where $\delta^{mi}$ is the
Kronecker tableau given by
$(\delta^{mi})_{kj}=\delta_{mk}\delta_{ij}$ and
\begin{equation}\label{eq:ami_def}
a_{mi}^\pm(\mu):=\mp\frac{\prod_{j=1}^{m\pm 1}[\tmu_{m\pm
1,j}-\tmu_{mi}]_q}{\prod_{j\in\{1,\ldots,m\}\setminus\{i\}}
[\tmu_{mj}-\tmu_{mi}]_q},
\end{equation}
where $\tmu_{mi}:=\mu_{mi}-i$ for all $1\le i\le m\le N$.
\end{thm}
Note that the denominator in \eqref{eq:ami_def} is always nonzero since $\la$ is admissible.
The following result will also be used.

\begin{thm}[\cite{MT}, in Proof of Theorem 4]
\label{MTthm2} The intersection of all annihilators of the
$U_q(\mathfrak{gl}_N)$-modules $V(\mu)$ as $\mu$ ranges over all
admissible tableaux, is zero.
\end{thm}

For $1\le m\le N$, put $U_m=U_q(\mathfrak{gl}_m)$. Denote by
$Z_m=Z(U_m)$ the center of the algebra   $U_m$.
Let $\Gamma$ be the \emph{Gelfand-Tsetlin subalgebra} of $U_N$
generated by $Z_1,\ldots, Z_N$.

 A finitely generated $U_q(\mf{gl}_N)$-module $M$ is called a \emph{Gelfand-Tsetlin module} if
\begin{equation}\label{eq:GTsMod}
M=\bigoplus_{\mf{m}\in\Specm \Ga} M(\mf{m}),
\end{equation}
where $M(\mf{m})=\big\{x\in M\mid \mf{m}^k x=0 \;\text{for some
$k\ge 0$} \big\}$ and $\Specm \Ga$ denotes the set of maximal
ideals in $\Ga$. The following result shows that the terminology
is sensible.
\begin{lem} \label{diagonal1}
For any admissible tableau $\la$, the generic Gelfand-Tsetlin module $V(\la)$ is a Gelfand-Tsetlin module. Moreover,
$\Gamma$ acts diagonally in the basis $B(\la)$ of a generic Gelfand-Tsetlin module $V(\la)$.
\begin{proof}
By \cite[Thm.~2]{MT}, $V(\la)$ has finite length and is therefore
finitely generated. That $\Gamma$ acts diagonally in the basis
$B(\la)$ follows from \cite[Proof of Theorem 2]{MT}. In particular
$V(\la)$ has a decomposition of the form \eqref{eq:GTsMod} and
thus is a Gelfand-Tsetlin module.
 \end{proof}
\end{lem}

\subsection{Realization of \texorpdfstring{$U_q(\mathfrak{gl}_N)$}{Uq(glN)} as a Galois \texorpdfstring{$\Gamma$}{Gamma}-ring}

Let $U_N=U_q(\mathfrak{gl}_N)$ and $\Mscr=\Z^{N(N-1)/2}$ with
$\Z$-basis $\{\delta^{mi}\}_{1\le i\le m\le N-1}$. Let $\Gamma$ be
the Gelfand-Tsetlin subalgebra of $U_q(\mathfrak{gl}_n)$. Let
$\La=\C[X_{mi}^{\pm 1}\mid 1\le i\le m\le N]$ be a Laurent
polynomial algebra in $N(N+1)/2$ variables. The group $\Mscr$ acts
on $\La$ by $\delta^{mi}X_{kj}=q^{-\delta_{mk}\delta_{ij}}X_{kj}$
for all $1\le i\le m\le N-1$ and $1\le j\le k\le N$.
 Let $L$ be the field of fractions of $\La$. Let
 $S\subseteq \La$ be the multiplicative subset generated
 by $\{q^{2l}X_{mj}^2-q^{2k}X_{mi}^2\mid k,l\in\Z, 1\le i,j\le m,
 i\neq j \}$, and let $\La_S$ be the localization.
Then $S$ is $\Mscr$-invariant, thus $\Mscr$ acts also on $\La_S$.
  The skew monoid ring $\La_S\ast\Mscr$ acts on any generic Gelfand-Tsetlin module $V(\la)$ as follows:
\begin{equation}\label{rhoRep}
\begin{aligned}
\rho_\la:\La_S\ast\Mscr&\to \End\big(V(\la)\big), \\
\rho_\la(\delta^{mi})[\mu]&=[\mu+\delta^{mi}],\qquad
\forall 1\le i\le m\le N-1,\\
\rho_\la(X_{mi})[\mu]&=q^{\tmu_{mi}}[\mu],\qquad\forall 1\le i\le
m\le N,
\end{aligned}\end{equation}
for all $[\mu]\in B(\la)$. Note that action of $s^{-1}$ for $s\in S$ is well-defined since $\la$ is admissible.

\begin{lem} \label{diagonal2}
If $a\in\La_S\ast\Mscr$ acts diagonally in the basis $B(\la)$ of a generic Gelfand-Tsetlin module $V(\la)$ for some admissible tableaux $\la$, then $a\in\La_S$.
\end{lem}
\begin{proof} Follows from the fact that the set $\{m[\la]\}_{m\in\Mscr}$ is linearly independent over $\C$.\end{proof}

\begin{prp}
There exists an injective algebra homomorphism $\varphi:U_N\to
\La_S\ast \Mscr$ determined by
\begin{equation}
\varphi(E_m^\pm) = \sum_{i=1}^N (\pm \delta^{mi})A_{mi}^\pm,\qquad
\varphi(K_m) =  A_m^0 e
\end{equation}
where $\delta^{mi}\in\Mscr$ are the tableaux units, $e\in\Mscr$ is the neutral element, and $A_{mi}^\pm$, $A_m^0\in \La_S$ are given by
\begin{align}
A_{mi}^\pm &= \mp (q-q^{-1})^{-1\mp 1} \frac{\prod_{j=1}^{m\pm 1} \big(X_{m\pm 1,j}X_{mi}^{-1}-X_{m\pm 1,j}^{-1}X_{mi}\big)}{\prod_{j\in\{1,\ldots,m\}\setminus\{i\}} \big( X_{mj}X_{mi}^{-1}-X_{mj}^{-1}X_{mi} \big) },\\
\label{eq:Am0_def}
A_m^0 &= q^m\prod_{i=1}^m X_{mi} \prod_{i=1}^{m-1} X_{m-1,i}^{-1}.
\end{align}
\end{prp}
\begin{proof}
Let $T$ be the free associative unital $\C$-algebra generated by
$\{E_i^\pm,K_j^\pm,\mid i=1,\ldots, N-1; j=1,\ldots, N\}$. Let
$p:T\to U_q(\gl_n)$ denote the canonical projection
$E_i^\pm\mapsto E_i^\pm, K_j^\pm\mapsto K_j^{\pm 1}$. Let
$\psi:T\to \La_S\ast\Mscr$ be given by
\begin{equation}
\psi(E_m^\pm) = \sum_{i=1}^N (\pm \delta^{mi})A_{mi}^\pm,\qquad
\psi(K_m^\pm) = (A_m^0)^{\pm 1}e.
\end{equation}
Let $\la$ be an admissible tableaux, $V(\la)$ the corresponding
generic Gelfand-Tsetlin module over $U_q(\mathfrak{gl}_N)$, and
$\tau_\la:U_q(\mathfrak{gl}_N)\to \End\big(V(\la)\big)$ the
associated representation. Recall the representation $\rho_\la$
from \eqref{rhoRep}. Note that algebra homomorphisms
$\rho_\la\circ \psi$ and $\tau_\la\circ p$ coincide on the
generators of $T$, hence they coincide on all of $T$. Let $V$ be
the direct product of all $V(\la)$ as $\la$ runs through the set
of all admissible tableaux. Thus $V$ is the set of families
$(v_\la)_{\la}$ indexed by admissible tableaux $\la$ and where
$v_\la\in V(\la)$ are arbitrary, not necessarily only finitely
many nonzero. Let $\tau:U_q(\mf{gl}_N)\to\End(V)$ and
$\rho:\La_S\ast\Mscr\to \End(V)$ be the respective product
representations. The two key points now are that $\rho\circ
\psi=\tau\circ p$ (since they are component-wise equal) and that,
by Theorem \ref{MTthm2}, $\tau$ is injective. These facts and a
quick diagram-chasing in Figure \ref{fig1} imply that
$\ker(\psi)\subseteq \ker(p)$. Thus, since $p$ is surjective, we
get an induced map $\varphi:U_q(\mathfrak{gl}_N)\to\La_S\ast\Mscr$
defined by $\varphi(a)=\psi(p^{-1}(a))$, which is the required
map.
\begin{figure}
\begin{equation}\label{CommDiag}
\begin{aligned}
\xymatrix@C=1.5cm{ 
 T \ar@{->>}[r]^{p} \ar[d]^{\psi}  &  U_q(\mathfrak{gl}_N) \ar@{^{(}->}[d]^{\tau} \ar@{-->}[dl]_{\varphi}\\
 \La_S\ast\Mscr   \ar[r]^{\rho } & \End(V)
}
\end{aligned}
\end{equation}
\caption{A commutative diagram.}\label{fig1}
\end{figure}
Furthermore, $\varphi$ is injective. Indeed, assume that $\varphi(a)=0$. Thus $\rho\circ\varphi(a)=0$. By the commutativity of \eqref{CommDiag}, we get $\rho\circ\varphi(a)=\tau(a)$. Since $\tau$ is injective, this implies that $a=0$.
\end{proof}

Let $W_N$ be the Weyl group of type $D_N$,
$W_N=S_N\ltimes \mathscr{E}_N$. Let $G=\prod_{m=1}^N W_m$.
Then $G$ acts on $\La$ by
\begin{subequations}\label{eq:G_action_on_skew_monoid}
\begin{equation}\label{eq:G_action_on_skew_monoid_1}
g(X_{mi})=(-1)^{\al_{mi}} X_{m \zeta_m(i)},\quad 1\le i\le m\le n,
\end{equation}
for $g=(\zeta_1\al_1,\cdots \zeta_N\al_N)\in G$ where $\zeta_m\in
S_m$, $\al_m=(\al_{m1},\ldots,\al_{mm})\in\mathscr{E}_m$. Note
also that $S$ is a $G$-invariant set, thus $G$ acts also on
$\La_S$. Viewing $\Mscr$ as a subset of $\End(\La_S)$, $G$ acts
naturally on $\Mscr$ by conjugations.
Explicitly,
\begin{equation}\label{eq:G_action_on_skew_monoid_2}
g(\delta^{mi}) = \delta^{m\zeta_m(i)}, \quad 1\le i\le m\le n-1,
\end{equation}
\end{subequations}
for $g=(\zeta_1\al_1,\cdots \zeta_N\al_N)\in G$. Note that the
subgroups $\mathcal{E}_m$ act trivially on $\Mscr$ for any
$m=1,\ldots,n$. Hence $G$ acts on the skew group ring
$\La_S\ast\Mscr$ by $\C$-algebra automorphisms.

\begin{prp}
 $\im \varphi\subseteq (\La_S\ast\Mscr)^G$.
\end{prp}
\begin{proof}
  By definition of $\varphi$, this is equivalent to showing
  that $\im \psi \subseteq (\La_S\ast\Mscr)^G$ for $\psi$ defined above. Since
  $(\La_S\ast\Mscr)^G$ is an algebra, it is enough to show that
  $\psi(a)\in (\La_S\ast\Mscr)^G$ for all $a$ in a generating
  set of $T$. We claim that $\psi(E_m^\pm)=[\delta^{m1}A_{m1}^\pm]$
  with notation as in \eqref{eq:GinvElt}. Indeed,
  $G/\Stab_G(\delta^{m1})\simeq \Z/m\Z$ with a set of
  representatives in $G$ given by
  $\{ (1), (12)_m, (13)_m,\ldots (1m)_m\}$, where $(ij)_m\in G$
  is the element with the transposition $(ij)$ placed in the
  $m$:th factor of $G$ and identity elements in the other $N-1$
  places. It is easy to check that
  $(A_{m1}^\pm)^{(1i)_m}=A_{mi}^\pm$ from which the claim follows.
  By \cite[Lemma~2.1]{FO} it follows that
  $\psi(E_m^\pm)\in (\La_S\ast\Mscr)^G$. It is visible
  from \eqref{eq:Am0_def} that the copy of $S_k$ in $G$
  acts trivially on $\psi(K_m)$ for any $k,m=1,\ldots, N$.
  Likewise, any
  $\al=(\al_1,\ldots,\al_k)\in\mathcal{E}_k$ fixes $\psi(K_m)$
  since $(-1)^{\al_1+\cdots+\al_k}=1$.
\end{proof}

For $m\in\iv{1}{n}$, let $\La_m=\C[X_{m1}^{\pm 1},\ldots,X_{mm}^{\pm 1}]\subseteq \La$ and let $\xi_m:\La_m\to U_m^0$ be the isomorphism given by $\xi_m(X_{mi})=\widetilde{K}_i=q^{-i}K_i$ for all $i$. Note that $\xi_m$ commutes with the action of $W_m$, when the $W_m$-action on $U_m^0$ is defined as in Section \ref{sec:center}.

The following result shows that the restriction of $\varphi$ to $Z_m$ can be identified with the quantum Harish-Chandra homomorphism.

\begin{prp}\label{prp:HC}
$\varphi\big|_{Z_m} = \xi_m^{-1}\circ h_m$
\end{prp}
\begin{proof}
Let $M$ be a type $1$ finite-dimensional irreducible
representation of $U_N$. As is well-known, it has a
Gelfand-Tsetlin basis, see e.g. \cite{KS}. This means that the
action of $U_N$ on $M$ is given by the exact same formulas as the
generic Gelfand-Tsetlin modules, except that the action of
$E_i^\pm$ on a basis vector is zero if the result lies outside the
support. Thus, when $z\in Z_m$ acts on a basis vector $[\mu]$ of
$M$, the resulting expression will be the same as if $[\mu]$ were
a basis vector of a generic Gelfand-Tsetlin module. That is, they
are given by the same Laurent polynomial in $q^{\mu_{mi}}$. From
the generic case, we know that this Laurent polynomial is
$\varphi(z)$ evaluated by substituting $X_{mi}$ by
$q^{\tilde{\mu}_{mi}}$. From the finite-dimensional case we get
the polynomial $h_m(z)\in\C[K_1^\pm,\ldots,K_m^\pm]$ evaluated by
substituting $K_i$ by $q^{\mu_{mi}}$, $i\in\iv{1}{m}$. This proves
the claim.
\end{proof}

\begin{prp}\label{prp:KLG}
Let $K:=\Frac\big(\varphi(\Gamma)\big)$. Then
$K=L^G$.
\end{prp}
\begin{proof}
It follows from Proposition \ref{prp:HC} and Lemma \ref{lem:image_of_HC} that $\varphi(\Gamma)=\La^G$.
Thus $K=L^G$.
\end{proof}

\begin{prp}
\begin{enumerate}[{\rm (a)}]
\item $\Mscr$ is $K$-separating;
\item $K\subseteq L$ is a finite Galois  extension with Galois group $G$.
\end{enumerate}
\end{prp}
\begin{proof}
(a) That $\Mscr$ is $K$-separating is easily seen by acting with
$\Mscr$ on $X_{m1}^2+\cdots+X_{mm}^2\in\La^G\subseteq K$ for
$m\in\iv{1}{N-1}$ and using that $q$ is not a root of unity.

(b) Proposition~\ref{prp:KLG} gives $K=L^G$. The field extension $K\subseteq L$ is normal since $L$ is the splitting field of the following polynomial in $K[x]$:
\[p(x)=\prod_{m=1}^N (x^2-X_{m1}^2)\cdots (x^2-X_{mm}^2)(x-X_{m1}\cdots X_{mm}).\]
Thus, since $\chara K=0$, $K\subseteq L$ is a Galois extension.
\end{proof}

We are now ready to prove that $U_q(\mf{gl}_N)$ can be realized as
a Galois $\Gamma$-ring.
\begin{thm} \label{thm:UqGalois}
The image of $\varphi$ is a
Galois $\varphi(\Gamma)$-ring in $(L\ast \Mscr)^G$.
\end{thm}
\begin{proof}
Since we have proved that we have the required setup of Section
\ref{sec:GaloisRings}, then the claim follows from
Proposition~\ref{prp:GaloisRing} by taking $u_i$ to be the images
under $\varphi$ of the generators $E_i^\pm, K_j$ of $U_N$.
\end{proof}

\section{Proof of the quantum Gelfand-Kirillov conjecture}
\label{sec:proof_of_qGK}

 In this section we prove Theorem~I by showing that the quantum Gelfand-Kirillov
conjecture follows from a positive solution to the $q$-difference
Noether problem.

By Theorem~\ref{thm:UqGalois} we have
\begin{equation}\label{eq:qGKstep1}
\begin{gathered}
\Frac(U_N)\simeq\Frac\big((L\ast\Mscr)^G\big)\simeq
\big(\Frac(\La\ast\Mscr)\big)^G  \\
\simeq \Frac\Big(
\bigotimes_{m=1}^{N-1}
 \big(\Frac(\La_m\ast \Z^m)\big)^{W_m} \otimes
 (\Frac \La_N)^{W_N}\Big),
\end{gathered}
\end{equation}
 where $W_m$ is the Weyl group of type $D_m$,
 $\La_m=\C[X_{m1}^{\pm1},\ldots,X_{mm}^{\pm 1}]$ and $\otimes=\otimes_\C$.

\begin{lem}\label{lem:iso}
There is an algebra isomorphism
\[\iota:\C_q(\bar x,\bar y)\overset{\sim}{\to}\Frac(\La_m\ast\Z^m)\]
where $\bar x=(x_1,\ldots,x_m)$ and $\bar y=(y_1,\ldots,y_m)$,  uniquely defined by
\[x_i\mapsto X_{mi}^{-1},\quad y_i\mapsto X_{mi}^{-1}\delta^{mi},\quad\forall i\in\iv{1}{n}.\]
Moreover, this isomorphism commutes with the $W_m$-action defined on both sides.
\end{lem}
\begin{proof}
We have $[X_{mi},X_{mj}]=0=[\delta^{mi},\delta^{mj}]$ for any
$i,j\in\iv{1}{m}$. By the definition of the action of $\Mscr$ on
$\La$ we have the commutation relation
$\delta^{mi}X_{mj}=q^{-\delta_{ij}}X_{mj}\delta^{mi}$, hence
$X_{mj}^{-1}\delta^{mj} X_{mi}^{-1}=q^{\delta_{ij}} X_{mi}^{-1}X_{mj}^{-1}\delta^{mj}$
 for all
$i,j\in\iv{1}{m}$. Since $y_jx_i=q^{\delta_{ij}}x_iy_j$,
 this proves that the map $\iota$ is well-defined, and
is clearly bijective. That it intertwines the $W_m$-actions is clear by the
definitions, \eqref{eq:Wm_action} and \eqref{eq:G_action_on_skew_monoid},
of the respective $W_m$-actions.
\end{proof}

Hence Lemma~\ref{lem:iso} reduces the quantum Gelfand-Kirillov
conjecture for $\mf{gl}_N$ to the $q$-difference Noether problem
for $W_N$. By Theorem~\ref{thm:qNoether_D} the right hand side
of \eqref{eq:qGKstep1} is isomorphic to
\begin{equation}
\Frac\Big( \bigotimes_{m=1}^{N-1}
 \Frac\big(\C_q[x,y]\otimes_\C \C_{q^2}[x,y]^{\otimes_\C (m-1) }\big)
 \otimes_\C \Frac(\La_N)^{W_N}\Big).
\end{equation}
Since $W_N$ is the Weyl
group of type $D_N$, it is in particular a complex reflection
group. Thus, by the Chevalley-Shephard-Todd theorem,
$\C[X_{N1},\ldots,X_{NN}]^{W_N}$ is a polynomial algebra in $N$ variables. Hence
$\Frac(\La_N)^{W_N}$ is isomorphic to a field $\K=\C(Z_1,\ldots,Z_N)$ of rational
functions in $N$ variables over $\C$. Thus
\begin{multline}
\Frac\Big( \bigotimes_{m=1}^{N-1}
 \Frac\big(\C_q[x,y]\otimes_\C \C_{q^2}[x,y]^{\otimes_\C (m-1)}\big)
 \otimes \Frac(\La_N)^{W_N} \Big) \\ \simeq
 \Frac\big(\K_q[x,y]^{\otimes_\K (N-1)}\otimes_\K \K_{q^2}[x,y]^{\otimes_\K (N-1)(N-2)/2)}\big)
\end{multline}
where $\K=\C(Z_1,\ldots,Z_N)$.
%
The proof of Theorem~I is completed.

\section{The quantum Gelfand-Kirillov conjecture for \texorpdfstring{$U_q^{\mathrm{ext}}(\mf{sl}_N)$}{Uqext(slN)}}
Let $U_q(\mf{sl}_N)$ be the quantized enveloping algebra of $\mf{sl}_N$ \cite{KS}.
The \emph{extented quantum group} $U_q^{\mathrm{ext}}(\mf{sl}_N)$ can be defined as the quotient of $U_q(\mf{gl}_N)$ by the ideal $\langle K_1K_2\cdots K_N-1\rangle$ (see
 \cite[Sec.~8.5.3]{KS}). Denoting the images of $E_i^{\pm}$ and $K_j$ by $E_i^{\pm}$ and $\widehat{K}_j$ respectively, there is an embedding
\begin{equation}
U_q(\mf{sl}_N)\longrightarrow U_q^{\mathrm{ext}}(\mf{sl}_N)
\end{equation}
given by the usual embedding $U_q(\mf{sl}_N)\to U_q(\mf{gl}_N)$ followed by the canonical projection. That is,
\begin{align*}
E_i^\pm &\longmapsto E_i^\pm,\\
K_i &\longmapsto \widehat{K}_i\widehat{K}_{i+1}^{-1},
\end{align*}
for $i\in\iv{1}{N-1}$.
Moreover, as is observed in \cite[Sec.~8.5.3]{KS},  $U_q^{\mathrm{ext}}(\mf{sl}_N)$ is isomorphic to the algebra obtained from $U_q(\mf{sl}_N)$ by adjoining the $N$:th roots
\begin{equation}\label{eq:Nth_root}
(K_1K_2^2\cdots K_{N-1}^{N-1})^{\pm 1/N}.
\end{equation}
The isomorphism maps $E_i$ to $E_i$ and $\widehat{K}_i$ to $K_i$ for $i\in\iv{1}{N-1}$ and maps $\widehat{K}_N$ to the element \eqref{eq:Nth_root}.

The following result shows that the quantum Gelfand-Kirillov conjecture holds for $U_q^{\mathrm{ext}}(\mf{sl}_N)$.
\begin{thm} \label{thm:qGK_ext}
There exists a $\C$-algebra isomorphism
\begin{equation} \label{eq:qGK_ext}
\Frac\big(U_q^{\mathrm{ext}}(\mf{sl}_N)\big)  \\ \simeq
 \Frac \Big( \K_q[x,y]^{\otimes_\K (N-1)}\otimes_\K
  \K_{q^2}[x,y]^{\otimes_\K (N-1)(N-2)/2} \Big)
\end{equation}
where $\K=\C(Z_1,\ldots,Z_{N-1})$.
\end{thm}
\begin{proof}
The element $K_1K_2\cdots K_N$ is a central element of $U_N$ and,
by Proposition \ref{prp:HC},
\[\varphi(K_1K_2\cdots K_N) = q^{N(N+1)/2}X_{N1}X_{N2}\cdots X_{NN}\in (\La_S\ast \mathcal{M})^G.\]
Therefore, the result follows by the isomorphisms in Section \ref{sec:proof_of_qGK},
by using that $q^{N(N+1)/2}X_{N1}X_{N2}\cdots X_{NN}$ can be taken as one
of the algebraically independent generators of $\C[X_{N1},\ldots,X_{NN}]^{W_N}$
and thus that
\[\La_N^{W_N}/\langle q^{N(N+1)/2}X_{N1}X_{N2}\cdots X_{NN}-1\rangle\simeq \C[Z_1,\Z_2,\ldots,Z_{N-1}].\]
\end{proof}

\subsection{Alev and Dumas' result for \texorpdfstring{$\mf{sl}_3$}{sl3}}
Recall the multiparameter quantized Weyl algebras $A_n^{\bar
q,\Lambda}(\K)$ from Section \ref{sec:qWskewfields}. In
\cite[Sec.~4.4]{AD}, the authors define a certain algebra, denoted
${U}_q^{AD}(\mf{sl}_3)$, and prove in \cite[Thm.~4.6]{AD} that
\begin{equation}\label{eq:AD_iso}
\Frac\big({U}_q^{AD}(\mf{sl}_3)\big)\simeq \Frac\Big(A_3^{\bar
q,\Lambda}\big(\C(Z_1,Z_2)\big)\Big),
\end{equation}
where $\bar q=(q,q,q^4)$ and $\La=(\la_{ij})$ with $\la_{ij}=1$ for all $i,j$ and $\C(Z_1,\Z_2)$ is the field of rational functions in two variables.
Following \cite{KS}, let $\breve{U}_{q^2}(\mf{sl}_3)$ denote the algebra
with generators $K_1^{\pm 1},K_2^{\pm 1}$, $E_1^\pm, E_2^\pm$ and relations
\begin{gather*}
K_iK_i^{-1}=K_i^{-1}K_i=1, \quad [K_i,K_j]=0,\quad\forall i,j\in\{1,2\},\\
\begin{aligned}
K_iE_j^\pm K_i^{-1} &= q^{\pm a_{ij}} E_j^{\pm}, \quad\forall i,j\in\{1,2\},\\
[E_i^+,E_j^-]&=\delta_{ij}\frac{K_i^2-K_i^{-2}}{q^2-q^{-2}}, \quad\forall i,j\in\{1,2\},\\
[E_i^{\pm},E_j^{\pm}]&=0,\quad |i-j|>1,
\end{aligned}\\
(E_i^\pm)^2E_j^\pm -(q^2+q^{-2})E_i^\pm E_j^\pm E_i^\pm + E_j^\pm (E_i^\pm)^2 =0, \quad |i-j|=1.
\end{gather*}
where $(a_{ij})=\left[\begin{smallmatrix} 2 &-1\\ -1&
2\end{smallmatrix}\right]$ is the Cartan matrix of $\mf{sl}_3$.
Alev and Dumas' algebra ${U}_q^{AD}(\mf{sl}_3)$ is obtained from
$\breve{U}_{q^2}(\mf{sl}_3)$ by adjoining $(K_1^2K_2)^{\pm 1/3}$. By
viewing $U_{q^2}^{\mathrm{ext}}(\mf{sl}_3)$ as an extension of
$U_{q^2}(\mf{sl}_3)$, we observe that there is a homomorphism
\begin{align*}
U_{q^2}^{\mathrm{ext}}(\mf{sl}_3) & \longrightarrow {U}_q^{AD}(\mf{sl}_3) \\
E_i^\pm &\longmapsto E_i^\pm, \quad i\in\{1,2\},\\
K_i & \longmapsto K_i^2,\quad i\in\{1,2\},\\
(K_1K_2^2)^{1/3} &\longmapsto (K_1^2K_2)^{1/3}\cdot K_2,
\end{align*}
Therefore we may equivalently view ${U}_q^{AD}(\mf{sl}_3)$ as
being obtained from $U_{q^2}^{\mathrm{ext}}(\mf{sl}_3)$ by
adjoining $K_1^{1/2}$ and $K_2^{1/2}$.

So let us define ${U}_q^{AD}(\mf{sl}_N)$ for general $N$ as the
algebra obtained from $U_{q^2}^{\mathrm{ext}}(\mf{sl}_N)$ by
adjoining $K_j^{1/2}$ for $j\in\iv{1}{N-1}$. By Proposition
\ref{prp:HC},
\[\varphi(K_1K_2\cdots K_m) = q^{m(m+1)/2}X_{m1}X_{m2}\cdots X_{mm}\]
for any $m\in\iv{1}{N}$. Furthermore, the isomorphism in Theorem \ref{thm:qNoether_main_result} maps $\widehat{X}_1=e_n=x_1x_2\cdots x_n$
to $x_1\in\K_q(\bar x,\bar y)$. Following through the isomorphisms, this means
that for $m\in\iv{1}{N-1}$, $K_1\cdots K_m$ is mapped under the map
\begin{equation}
\Frac\big(U_{q^2}^{\mathrm{ext}}(\mf{sl}_N)\big)  \overset{\sim}{\longrightarrow}
 \Frac \Big( \K_{q^2}[x,y]^{\otimes_\K (N-1)}\otimes_\K
  \K_{q^4}[x,y]^{\otimes_\K (N-1)(N-2)/2} \Big)
\end{equation}
to some nonzero $\K$-multiple of the element
\[ x_m = 1^{\otimes (m-1)}\otimes x\otimes 1^{\otimes (N-1)(N-2)/2-m}. \]
Therefore, adjoining the square roots $K_j^{1/2}$ for $j\in\iv{1}{N-1}$,
or equivalently $(K_1K_2\cdots K_j)^{1/2}$ for $j\in\iv{1}{N-1}$, to
 $U_{q^2}^{\mathrm{ext}}(\mf{sl}_N)$, corresponds to adjoining
 the square roots $x_m^{1/2}$ for $m=1,\ldots,N-1$. This shows that
\begin{equation}
\Frac\big({U}_q^{AD}(\mf{sl}_N)\big) \simeq
 \Frac \Big( \K_{q}[x,y]^{\otimes_\K (N-1)}\otimes_\K
  \K_{q^4}[x,y]^{\otimes_\K (N-1)(N-2)/2} \Big)
\end{equation}
In particular, for $N=3$ we recover \eqref{eq:AD_iso}, bearing in mind
Proposition~\ref{prp:qWeyl_iso}.

\section{Appendix}

\subsection{Proof of Proposition~\ref{prp:tcommutes}}

The statement is equivalent to proving that
\begin{equation}\label{eq:PXPY}
[P(X),P(Y)]=0
\end{equation}
in $C_n^q[X,Y]$. Put
\begin{equation}
Q_j(X)= \left(\prod_{k\in\{1,\ldots,n\}\setminus\{j\}} \frac{X-x_k}{x_j-x_k}\right)y_j.
\end{equation}
so that $P(X)=\sum_{i=1}^n Q_i(X)$.
Observe that
\begin{equation}\label{eq:Qsymmetry}
w\big(Q_j(X)\big) = Q_{w(j)}(X),\qquad w\in S_n.
\end{equation}
Thus, to prove \eqref{eq:PXPY}, it is enough to show the following two identities:
\begin{align}
\label{eq:Qfirst}
[Q_1(X),Q_1(Y)]&=0,\\
\label{eq:Qsecond}
[Q_1(X),Q_2(Y)]+[Q_2(X),Q_1(Y)]&=0.
\end{align}
Since $y_1 x_i = q^{\delta_{i1}} x_i y_1$ we have
\begin{align*}
Q_1(X)Q_1(Y)&=
\prod_{k=2}^n \frac{X-x_k}{x_1-x_k}y_1 \prod_{k=2}^n\frac{Y-x_k}{x_1-x_k}y_1=\\
&=\prod_{2\le k\le n} \frac{(X-x_k)(Y-x_k)}{(x_1-x_k)(q
x_1-x_k)}y_1^2
\end{align*}
which is symmetric in $X,Y$. This proves \eqref{eq:Qfirst}.

Next we prove \eqref{eq:Qsecond}.
Let
\[R_j(X)=\prod_{k=3}^n \frac{X-x_k}{x_j-x_k},\qquad j=1,2.\]
Then
\[Q_1(X)=\frac{X-x_2}{x_1-x_2} R_1(X) y_1,\qquad Q_2(X)=\frac{X-x_1}{x_2-x_1} R_2(X) y_2,\]
\[[y_1, R_2(X)]=[y_2,R_1(X)]=0, \]
and
\[R_1(X)R_2(Y)=R_1(Y)R_2(X)=R_2(X)R_1(Y)=R_2(Y)R_1(X).\]
We have
\begin{align*}
&[Q_1(X),Q_2(Y)]+[Q_2(X),Q_1(Y)]=
Q_1(X)Q_2(Y)-Q_1(Y)Q_2(X)\\ &\quad+ Q_2(X)Q_1(Y) - Q_2(Y)Q_1(X) \\
&=\frac{X-x_2}{x_1-x_2}  \cdot \frac{Y-qx_1}{x_2-qx_1}
R_1(X)R_2(Y) y_1y_2 -\frac{Y-x_2}{x_1-x_2}  \cdot
\frac{X-qx_1}{x_2-qx_1} R_1(Y)R_2(X) y_1y_2
\\
&\quad+
\frac{X-x_1}{x_2-x_1} \cdot
\frac{Y-qx_2}{x_1-qx_2} R_2(X)R_1(Y) y_1y_2
-
\frac{Y-x_1}{x_2-x_1} \cdot
\frac{X-qx_2}{x_1-qx_2} R_2(Y)R_1(X) y_1y_2
 \\
&=\Bigg(\frac{(X-x_2)(Y-qx_1)-(Y-x_2)(X-qx_1)}{(x_1-x_2)(x_2-qx_1)}\\
&\qquad\qquad+
\frac{(X-x_1)(Y-qx_2)-(Y-x_1)(X-qx_2)}{(x_2-x_1)(x_1-qx_2)}
 \Bigg) R_1(X)R_2(Y) y_1y_2 \\
&=\Bigg(\frac{(x_2^2-q^2x_1^2)X -
(x_2-qx_1)Y}{(x_1-x_2)(x_2-qx_1)} +
\frac{(x_1-qx_2)X-(x_1-qx_2)Y}{(x_2-x_1)(x_1-qx_2)}\Bigg)
R_1(X)R_2(Y)y_1y_2
\\&=0
\end{align*}
This shows \eqref{eq:Qsecond} and completes the proof that $[t_i,t_j]=0$ for all $i,j$.

\subsection{Proof of Proposition~\ref{prp:tj_ek}}

The relation \eqref{eq:ti_tj} holds by Proposition
\ref{prp:tcommutes}, while \eqref{eq:ek_el} holds by the
definition, \eqref{eq:ed_def}, of $e_d$. Relation \eqref{eq:tj_ek}
is trivial for $k=0$.

Using \eqref{eq:formula_for_tj} and that $w(e_k)=e_k$ for any
$w\in S_n$ we have, for any $j,k\in\{1,\ldots,n\}$,
\[ 
(-1)^{j-1}\Delta\cdot t_je_k=\sum_{w\in S_n}
\sgn(w)w\Big(x_1^{n-2}x_2^{n-3}\cdots x_{n-2}
e_{n-j}'e_k(x_1,\ldots,x_{n-1},qx_n)y_n\Big)
\] 
Substituting $y_n=t_1+x_nt_2+\cdots+x_n^{n-1}t_n$ and using that
$w(t_i)=t_i$ for all $w\in S_n$ we get
$$ (-1)^{j-1}\Delta\cdot t_je_k =
$$ $$ \sum_{i=1}^n\sum_{w\in S_n}
\sgn(w)w\Big(x_1^{n-2}x_2^{n-3}\cdots x_{n-2} \cdot
x_n^{i-1}e_{n-j}'e_k(x_1,\ldots,x_{n-1},qx_n)\Big)t_i.$$
Write $e_{n-j}'$ as a sum of monomials $x_{i_1}\cdots x_{i_{n-j}}$
and $1\le i_1<\cdots <i_{n-j}\le n-1$. We claim that the only way
to get a nonzero contribution is when $i_r=r$ for all $r$. Indeed,
suppose $i_r>r$ for some $r$ chosen minimal. Then the product
\[x_1^{n-2}x_2^{n-3}\cdots x_{n-2} \cdot x_{i_1}\cdots x_{i_{n-j}}\cdot x_n^{i-1}e_k(x_1,\ldots,x_{n-1},qx_n)\]
will be
fixed by the transposition $(i_r-1\;\;i_r)$. Therefore, after
anti-symmetrization, the term will cancel out. In other words,
the substitution $w\mapsto w\cdot(i_r-1\;\; i_r)$ in the sum
\[\sum_{w\in S_n} \sgn(w)w\Big(x_1^{n-2}x_2^{n-3}\cdots x_{n-2} \cdot x_{i_1}\cdots x_{i_{n-j}}\cdot x_n^{i-1}e_k(x_1,\ldots,x_{n-1},qx_n)\Big)t_i\]
gives the same expression with opposite sign, proving it is zero.
Thus, noting also that \[e_k(x_1,\ldots,x_{n-1},qx_n)=e_k' + qx_n
e_{k-1}',\] we have
\begin{multline}\label{eq:tj_ek_second_step}
(-1)^{j-1}\Delta\cdot t_je_k\\ =\sum_{i=1}^n\sum_{w\in S_n}
\sgn(w)w\Big(x_1^{n-1}\cdots x_{n-j}^j\cdot x_{n-j+1}^{j-2}\cdots
x_{n-2} x_n^{i-1} (e_k'+qx_ne_{k-1}')\Big)t_i.
\end{multline}

\paragraph{\bf The term $i=j$:}
Write $e_k'= \sum_{1\le i_1<\cdots <i_k\le n-1} x_{i_1}\cdots
x_{i_k}$. Consider
\[x_1^{n-1}\cdots x_{n-j}^j x_{n-j+1}^{j-2}\cdots x_{n-2} x_n^{j-1} \cdot x_{i_1}\cdots x_{i_k}\]
An expression like this containing factors $(x_rx_{r'})^s$ ($r\neq
r'$) will become zero after anti-symmetrization. If $n-j\ge k$
there is a unique way to get a nonzero result, namely to choose
$(i_1,\ldots,i_k)=(1,2,\ldots,k)$. If $n-j<k$ there is no way to
get nonzero result. Thus
\begin{align*}
&\sum_{w\in S_n} \sgn(w)w\Big(x_1^{n-1}\cdots x_{n-j}^j\cdot x_{n-j+1}^{j-2}\cdots x_{n-2} x_n^{j-1}e_k'\Big)  \\
&=\begin{cases}
a(n,n-1,\ldots,n-k+1,n-k-1,\ldots,j,j-2,\ldots,1,0,j-1),& j+k\le n\\
0,& j+k>n
\end{cases}
\end{align*}
where $a(i_1,\ldots,i_n):=\sum_{w\in S_n}\sgn(w)w(x_1^{i_1}\cdots
x_n^{i_n})$. Use that
$w(a(i_1,\ldots,i_n))=\sgn(w)a(i_1,\ldots,i_n)$ with
\[w=(n-j+1\quad n-j+2\quad\cdots\quad n),\]
which is a cycle of length $j$, to get
\begin{multline*}
a(n,n-1,\ldots,n-k+1,n-k-1,\ldots,j,j-2,\ldots,1,0,j-1)\\
=(-1)^{j-1}a(n,n-1,\ldots,n-k+1,n-k-1,\ldots,0).
\end{multline*}
Using that the Schur function
\[s_\la = a(\la_1+n-1,\la_2+n-2,\ldots,\la_n)/a(n-1,n-2,\ldots,0),\]
defined for a partition $\la=(\la_1,\ldots,\la_n)$,
$\la_1\ge\cdots\ge\la_n\ge 0$, satisfies $s_{1^k0^{n-k}}=e_k$ and
that $\Delta=a(n-1,n-2,\ldots,0)$ we get that
\begin{multline}
\sum_{w\in S_n} \sgn(w)w\Big(x_1^{n-1}\cdots x_{n-j}^j\cdot x_{n-j+1}^{j-2}\cdots x_{n-2} x_n^{j-1}e_k'\Big)  \\
=
\begin{cases}
(-1)^{j-1} \Delta\cdot e_k, & j+k\le n,\\
0,& j+k>n.
\end{cases}\label{eq:tj_ek_case_i_equals_j_1}
\end{multline}
Similarly, if we look at the term containing $qx_ne_{k-1}'$, there
is at most one tuple $(i_1,\ldots,i_{k-1})$, $1\le i_1<\cdots
<i_{k-1}\le n-1$ such that the antisymmetrization of
\[qx_1^{n-1}\cdots x_{n-j}^j x_{n-j+1}^{j-2}\cdots x_{n-2}x_n^j x_{i_1}\cdots x_{i_{k-1}}\]
is nonzero, namely $(i_1,\ldots,i_{k-1})=(1,\ldots,k-1)$ and this
time, due to the presence of $x_n^j$, it gives nonzero result if
and only if $k-1\ge n-j$ i.e. $j+k>n$. Thus
\begin{align*}
&\sum_{w\in S_n} \sgn(w)w\Big(x_1^{n-1}\cdots x_{n-j}^j\cdot x_{n-j+1}^{j-2}\cdots x_{n-2} \cdot qx_n^je_{k-1}'\Big) \\
&=\begin{cases}
0,& j+k\le n\\
qa(n,n-1,\ldots,j+1,j-1,\ldots,n-k+1,n-k-1,\ldots,1,0,j),& j+k>n
\end{cases}
\end{align*}
To get a descending sequence inside the parenthesis we apply the
cyclic permutation which places $j$ between $j-1$ and $j+1$. This
cycle has length $j$, giving a factor $(-1)^{j-1}$. As before,
this gives
\begin{multline}
\sum_{w\in S_n} \sgn(w)w\Big(x_1^{n-1}\cdots x_{n-j}^j\cdot x_{n-j+1}^{j-2}\cdots x_{n-2} \cdot qx_n^je_{k-1}'\Big) \\
=
\begin{cases}
0,& j+k\le n\\
(-1)^{j-1}q \Delta\cdot e_k, & j+k> n,\\
\end{cases}\label{eq:tj_ek_case_i_equals_j_q}
\end{multline}
Combining \eqref{eq:tj_ek_case_i_equals_j_1} and
\eqref{eq:tj_ek_case_i_equals_j_q} yields
\begin{multline}\label{eq:tj_ek_third_step}
(-1)^{j-1}\Delta\cdot (t_je_k - q^{\delta_{j+k>n}} e_kt_j) \\
= \sum_{i\in\{1,\ldots,n\}\setminus\{j\}}\sum_{w\in S_n}
\sgn(w)w\Big(x_1^{n-1}\cdots x_{n-j}^j\cdot x_{n-j+1}^{j-2}\cdots
x_{n-2} x_n^{i-1} (e_k'+qx_ne_{k-1}')\Big)t_i.
\end{multline}

\paragraph{\bf The terms where $i>j$:}
We first look at the $e_k'$ term in \eqref{eq:tj_ek_third_step}.
That $i>j$ means the exponent $i-1$ of $x_n$ occurs in one of the
exponents in $x_1^{n-1}x_2^{n-2}\cdots x_{n-j}^j$, namely in
$x_{n-(i-1)}^{i-1}$. Therefore $x_{i_1}\cdots x_{i_k}$ must
contain $x_1x_2\cdots x_{n-(i-1)}$. In particular $k\ge n-(i-1)$.
The remaining factors must be $x_{n-j+1}x_{n-j+2}\cdots $ and they
cannot continue beyond $x_{n-1}$ meaning that $k-(n-i+1) + (n-j)
\le n-1$. Thus the following inequalities are necessary conditions
in order to avoid having two variables with the same exponent:
\[k\ge  n-i+1,\quad\text{ and }\quad k+i-j-1\le n-1,\]
i.e.
\[n-k+1\le i\le n-k+j.\]
If these inequalities hold there is a unique tuple
\[(i_1,\ldots,i_k)=(1,2,\ldots,n-i+1, n-j+1,n-j+2,\ldots,k+i-j-1)\]
with $1\le i_1<\cdots<i_k\le n-1$ such that
\[\sum_{w\in S_n} \sgn(w)w\Big(x_1^{n-1}\cdots x_{n-j}^j\cdot x_{n-j+1}^{j-2}\cdots x_{n-2}^2 x_n^{i-1} \cdot x_{i_1}\cdots x_{i_k}\Big).
\]
is nonzero. With this choice we get
\begin{align}
\sum_{w\in S_n} &\sgn(w)w\Big(x_1^{n-1}\cdots x_{n-j}^j\cdot
x_{n-j+1}^{j-2}\cdots x_{n-2} x_n^{i-1} \cdot x_{i_1}\cdots
x_{i_k}\Big)
\nonumber \\
&= a(n,\ldots,i,i-2,\ldots,n-(k+i-j)+1,n-(k+i-j)-1,\ldots,0,i-1)\nonumber \\
&= (-1)^i a(n,n-1,\ldots,n-(k+i-j)+1,n-(k+i-j)-1,\ldots,0) \nonumber \\
&= (-1)^i \Delta\cdot e_{k+i-j} \label{eq:tj_ek_i_bigger_than_j_1}
\end{align}
where we applied the cyclic permutation $(n-i+2\;\; n-i+3\;\;
\cdots\;\; n-1\;\; n)$ of length $i-1$ in the second equality.

The argument for the term containing $qx_1e_{k-1}'$ is analogous,
but gives an extra minus sign. Together with
\eqref{eq:tj_ek_i_bigger_than_j_1} one obtains that for $i>j$ we
have
\begin{multline}\label{eq:tj_ek_i_bigger_than_j}
\sum_{w\in S_n} \sgn(w)w\Big(x_1^{n-1}\cdots x_{n-j}^j\cdot
x_{n-j+1}^{j-2}\cdots x_{n-2} x_n^{i-1}
(e_k'+qx_ne_{k-1}')\Big)t_i
\\
=\begin{cases}
(-1)^{i+1}(q-1)\Delta \cdot e_{k+i-j}t_i,& n-k+1\le i\le n-k+j,\\
0,&\text{otherwise}.
\end{cases}
\end{multline}

\paragraph{\bf The terms where $i<j$:}
We look at the $e_k'$ term in \eqref{eq:tj_ek_third_step}.
Necessary conditions for nonzero contribution are $k\ge j-i$ and
$k-(j-i)\le n-j$, i.e.
\begin{equation}
j-k\le i \le n-k.
\end{equation}
After a similar computation as the $i>j$ case we obtain
\begin{multline}\label{eq:tj_ek_i_less_than_j}
\sum_{w\in S_n} \sgn(w)w\Big(x_1^{n-1}\cdots x_{n-j}^j\cdot
x_{n-j+1}^{j-2}\cdots x_{n-2} x_n^{i-1}
(e_k'+qx_ne_{k-1}')\Big)t_i
\\
=\begin{cases}
(-1)^i(q-1)\Delta\cdot e_{k+i-j}t_i , & j-k\le i\le n-k\\
0,& \text{otherwise}.
\end{cases}
\end{multline}
Combining \eqref{eq:tj_ek_i_less_than_j},
\eqref{eq:tj_ek_i_bigger_than_j} and \eqref{eq:tj_ek_third_step}
we obtain
\begin{multline}
t_je_k - q^{\delta_{j+k>n}}e_kt_j =(q-1)\sum_{\substack{i>j\\
n-k+1\le i\le n-k+j}}
(-1)^{j-1+i+1} e_{k+i-j} t_i \\
+ (q-1)\sum_{\substack{i<j\\ j-k\le i\le n-k}} (-1)^{j-1+i}
e_{k+i-j}t_i
\end{multline}
Making the change of summation variables $i\mapsto i+j$ we get
\begin{multline}
t_je_k - q^{\delta_{j+k>n}}e_kt_j =(q-1)\sum_{\substack{i>0\\
n-(j+k)+1\le i\le n-k}}
(-1)^{i+\delta_{i<0}} e_{k+i} t_{j+i} \\
+ (q-1)\sum_{\substack{i<0\\ -k\le i\le n-(j+k)}}
(-1)^{i+\delta_{i<0}} e_{k+i}t_{j+i}.
\end{multline}
In the first sum, the condition $i\le n-k$ is redundant since,
 by the notational convention, $e_{k+i}=0$ for $i>n-k$. Similarly, $-k\le i$
 is superfluous in the second sum. Thus we obtain \eqref{eq:tj_ek}.

\subsection{Proof of Proposition~\ref{prp:inductive_procedure}}

First note that \eqref{eq:TjEk_assumption} implies that
\begin{align}
\label{eq:TjE0}
[T_j,E_0] &= 0,\quad\forall j\in\iv{1}{n},\\
\label{eq:TjEn} [T_j, E_n]_q &=0,\quad\forall j\in\iv{1}{n}.
\end{align}
We now prove \eqref{eq:tilde_T_tilde_E}.
Let $j\in\iv{1}{n-1}$ and $k\in\iv{0}{n-1}$. Then the left hand
side of \eqref{eq:tilde_T_tilde_E} equals
\begin{align}
[\widetilde{E}_k, \widetilde{T}_j]_{q^{\delta_{j+k>n-1}}} =&
[T_{k+1}, E_jT_1T_n-(-1)^jE_0T_{n-j}T_1-(-1)^{n-j}E_nT_{n+1-j}T_n]_{q^{\delta_{j+1+k>n}}} \\
\label{eq:tilde_T_tilde_E_proof_0}
=&(q-1)\sum_{i\in\Z\setminus I(n-1-j-k)} (-1)^{i+\delta_{i<0}} E_{j+i} T_{k+1+i} T_1T_n \\
\label{eq:tilde_T_tilde_E_proof_1}
&-(1-q^{\delta_{j+k+1>n}})(-1)^jE_0T_{k+1}T_{n-j}T_1\\
\label{eq:tilde_T_tilde_E_proof_2}
&-(q-q^{\delta_{j+k+1>n}})(-1)^{n-j}E_nT_{k+1}T_{n+1-j}T_n.
\end{align}
By \eqref{eq:T_identity_1} with $(j,k)$ replaced by $(k+1,n-j)$,
the term \eqref{eq:tilde_T_tilde_E_proof_1} equals
\begin{equation} \label{eq:tTtEpf1}
-(q-1)\sum_{i\in\Z\setminus
I(n-1-j-k)}(-1)^{j+\delta_{i<0}}E_0T_{k+1+i}T_{n-j-i}T_1.
\end{equation}
Similarly, applying \eqref{eq:T_identity_2} with $(j,k)$ replaced
by $(k+1,n+1-j)$ shows that \eqref{eq:tilde_T_tilde_E_proof_2} is
equal to
\begin{equation}\label{eq:tTtEpf2}
-(q-1)\sum_{i\in\Z\setminus
I(n-1-j-k)}(-1)^{n-j+\delta_{i<0}}E_nT_{k+1+i}T_{n+1-j-i}T_n.
\end{equation}
Adding together \eqref{eq:tTtEpf1}, \eqref{eq:tTtEpf2} and
\eqref{eq:tilde_T_tilde_E_proof_0} gives the right hand side of
\eqref{eq:tilde_T_tilde_E}. This proves
\eqref{eq:tilde_T_tilde_E}.

In particular, taking $k=0$ and $k=n-1$ in
\eqref{eq:tilde_T_tilde_E} we get
\begin{align}
\label{eq:T_tilde_T_1}
T_1\widetilde{T}_j-\widetilde{T}_jT_1 &=0,\\
\label{eq:T_tilde_T_2} T_n\widetilde{T}_j-q\widetilde{T}_jT_n &=0,
\end{align}
for all $j\in\iv{1}{n-1}$. Using these identities, together with
$[T_j,E_0]=0$ and $[T_j,E_n]_q=0$ which follow from
\eqref{eq:TjEk_assumption}, one can check that
\[q\widetilde{T}_j=T_1T_n\widetilde{T}_j(T_1T_n)^{-1}=
T_nT_1E_j-(-1)^{j}T_1T_{n-j}E_0-(-1)^{n-j}T_nT_{n+1-j}E_n,\]
proving \eqref{eq:tilde_T_alt}.

That \eqref{eq:tilde_E_commute} holds is trivial from the
assumption \eqref{eq:TiTj_assumption}.

We now prove \eqref{eq:tilde_T_commute}. Let $j,k\in\iv{1}{n-1}$.
We will bring $\widetilde{T}_j\widetilde{T}_k$ to the normal form
where all the $E$'s are to the left of all the $T$'s and prove
that the resulting expression is symmetric in $j,k$. We may assume
$j\neq k$. Using \eqref{eq:T_tilde_T_1}, \eqref{eq:T_tilde_T_2}
and \eqref{eq:tilde_T_tilde_E}, we have
\begin{align}
\widetilde{T}_j\widetilde{T}_k =& (E_jT_1T_n-(-1)^jE_0T_{n-j}T_1-(-1)^{n-j}E_nT_{n+1-j}T_n)\widetilde{T}_k\nonumber\\
=&qE_j\widetilde{T}_kT_1T_n-(-1)^jE_0T_{n-j}\widetilde{T}_kT_1-(-1)^{n-j}qE_nT_{n+1-j}\widetilde{T}_kT_n\nonumber\\
=&qE_j\widetilde{T}_kT_1T_n\nonumber\\
&-(-1)^jE_0\Big(q^{\delta_{-j+k>0}}\widetilde{T}_k T_{n-j}+(q-1)\sum_{i\in\Z\setminus I(j-k)}(-1)^{i+\delta_{i<0}}\widetilde{T}_{k+i}T_{n-j+i}\Big)T_1\nonumber\\
&-(-1)^{n-j}qE_n\Big(q^{\delta_{1-j+k>0}}\widetilde{T}_kT_{n+1-j}+ \nonumber \\
&\qquad (q-1)\sum_{i\in\Z\setminus I(-1+j-k)} (-1)^{i+\delta_{i<0}} \widetilde{T}_{k+i} T_{n+1-j+i}\Big)T_n\nonumber\\
=&qE_jE_kT_1^2T_n^2 - (-1)^kqE_jE_0T_{n-k}T_1^2T_n -(-1)^{n-k}qE_jE_nT_{n+1-k}T_1T_n^2\nonumber\\
&-(-1)^jq^{\delta_{k>j}} E_0E_kT_{n-j}T_1^2T_n +(-1)^{j+k}q^{\delta_{k>j}}E_0^2T_1^2 T_{n-k}T_{n-j} + \nonumber\\
&\qquad\qquad +(-1)^{n-k+j}q^{\delta_{k>j}} E_0E_nT_{n-j}T_{n+1-k}T_1T_n\nonumber\\
&-(-1)^j(q-1)\sum_{i\in\Z\setminus I(j-k)}(-1)^{i+\delta_{i<0}} E_0\Big(E_{k+i}T_1T_n \nonumber\\
&\qquad\qquad-(-1)^{k+i}E_0T_{n-k-i}T_1-(-1)^{n-k-i}E_nT_{n+1-k-i}T_n\Big)T_{n-j+i}T_1\nonumber\\
&-(-1)^{n-j}q^{1+\delta_{k\ge j}}E_nE_k T_1T_n T_{n+1-j}T_n + (-1)^{n-j+k}q^{1+\delta_{k\ge j}}E_nE_0T_{n-k}T_1T_{n+1-j}T_n\nonumber\\
&\qquad\qquad+ (-1)^{2n-j-k}q^{1+\delta_{k\ge j}} E_n^2T_{n+1-k}T_n^2T_{n+1-j}\nonumber\\
&-(-1)^{n-j}q(q-1)\sum_{i\in\Z\setminus I(-1+j-k)} (-1)^{i+\delta_{i<0}}E_n\Big(E_{k+i}T_1T_n\nonumber\\
\label{eq:big_sum_of_stuff} &\qquad\qquad
-(-1)^{k+i}E_0T_{n-k-i}T_1-(-1)^{n-k-i}E_nT_{n+1-k-i}T_n\Big)T_{n+1-j+i}T_n.
\end{align}
We prove that all parts of this expression are symmetric in $j,k$.
The first term, containing $E_jE_k$, is trivially symmetric.\\

\paragraph{\bf The terms containing $E_0^2T_1^2$.}
There are two terms in \eqref{eq:big_sum_of_stuff} containing
$E_0^2T_1^2$:
\begin{equation}\label{eq:E02_step}
(-1)^{j+k}q^{\delta_{k>j}}E_0^2T_{n-k}T_{n-j}T_1^2  +
(-1)^{j+k}(q-1)\sum_{i\in\Z\setminus I(j-k)} (-1)^{\delta_{i<0}}
E_0^2T_1^2T_{n-k-i}T_{n-j+i}.
\end{equation}
Applying \eqref{eq:T_identity_1} with $(j,k)$ replaced by
$(n-j,n-k)$ we get that \eqref{eq:E02_step} equals
\[
(-1)^{j+k}E_0^2T_{n-j}T_{n-k}T_1^2
\]
which is symmetric in $j,k$.\\

\paragraph{\bf The terms containing $E_n^2T_n^2$.}
\begin{multline}\label{eq:En2_step}
(-1)^{2n-j-k}q^{1+\delta_{k\ge j}} E_n^2T_n^2T_{n+1-k}T_{n+1-j} \\
+(-1)^{2n-j-k}q(q-1)\sum_{i\in\Z\setminus I(-1+j-k)}
(-1)^{\delta_{i<0}} E_n^2T_n^2T_{n+1-k-i}T_{n+1-j+i}
\end{multline}
Here we can apply \eqref{eq:T_identity_2} with $(j,k)$ replaced by
$(n+1-j,n+1-k)$ to see that \eqref{eq:En2_step} equals
\[(-1)^{j+k}q^2 E_n^2T_n^2T_{n+1-j}T_{n+1-k}\]
which is symmetric in $j,k$.\\

\paragraph{\bf The terms containing $E_0T_1^2T_n$.}
\begin{multline}\label{eq:T12Tn_step}
-E_0\Big((-1)^kqE_jT_{n-k}
+(-1)^jq^{\delta_{k>j}} E_0E_k T_{n-j}\\
+(-1)^j(q-1)\sum_{i\in\Z\setminus I(j-k)} (-1)^{i+\delta_{i<0}}
E_0E_{k+i}T_{n-j+i}\Big)T_1^2T_n
\end{multline}
The parenthesis equals
\begin{multline}\label{eq:junk_237}
(-1)^kE_j T_{n-k} + (-1)^j E_kT_{n-j} + \\
(-1)^k(q-1)E_jT_{n-k} + (-1)^j(q^{\delta_{k>j}}-1)E_kT_{n-j}+\\
(-1)^j(q-1)\sum_{i\in\Z\setminus I(j-k)} (-1)^{i+\delta_{i<0}}
E_{k+i} T_{n-j+i}.
\end{multline}
If $j\ge k$, we can include $(-1)^k(q-1)E_jT_{n-k}$ as the term
$i=j-k$ in the sum. If $j<k$, the term $(-1)^k(q-1)E_jT_{n-k}$
cancels the term $i=j-k$ in the sum, and
$(-1)^j(q^{\delta_{k>j}}-1)E_kT_{n-j}$ may be included in the sum
as $i=0$. Thus \eqref{eq:junk_237} can be written
\begin{multline}
(-1)^kE_jT_{n-k}+(-1)^jE_kT_{n-j}\\
+(-1)^j(q-1)\sum_{\substack{i\in
\begin{cases}
\Z\setminus\iv{0}{j-k-1},& j\ge k\\
\Z\setminus\iv{j-k}{-1},& j<k
\end{cases}}}
 (-1)^{i+\delta_{i<0}} E_{k+i}T_{n-j+i}.
\end{multline}
Making the change of variables $i\mapsto i+j-k$ in this sum gives
the same expression but with $j$ and $k$ interchanged. Thus it is
symmetric in $j$ and $k$.\\

\paragraph{\bf The terms containing $E_nT_1T_n^2$.}
\begin{multline}
-qE_n\Big((-1)^{n-k}E_jT_{n+1-k}
+(-1)^{n-j}q^{\delta_{k\ge j}} E_k T_{n+1-j} \\
+(-1)^{n-j}(q-1)\sum_{i\in\Z\setminus I(-1+j-k)}
(-1)^{i+\delta_{i< 0}} E_{k+i} T_{n+1-j+i} \Big)T_1T_2^2
\end{multline}
Similarly to the previous case, the expression inside the
parenthesis can be written as
\begin{multline}
(-1)^{n-k}q E_jT_{n+1-k} + (-1)^{n-j} q E_kT_{n+1-j}\\
+(-1)^{n-j}(q-1)\sum_{\substack{i\in
\begin{cases}
\Z\setminus\iv{1}{j-k},&j>k\\
\Z\setminus\iv{j-k+1}{0},&j\le k
\end{cases}}}
(-1)^{i+\delta_{i\le 0}} E_{k+i}T_{n+1-j+i}.
\end{multline}
Substituting $i\mapsto i-k+j$ one checks this is symmetric in $j$
and $k$.\\

\paragraph{\bf The terms containing $E_0E_nT_1T_n$.}
Finally, there are four terms in \eqref{eq:big_sum_of_stuff}
containing $E_0E_nT_1T_n$:
\begin{multline}\label{eq:E0EnT1Tn_step}
E_0E_n\Big((-1)^{n-k+j}q^{\delta_{k> j}} T_{n+1-k}T_{n-j}
+ (-1)^{n-j+k}q^{1+\delta_{k\ge j}} T_{n-k}T_{n+1-j}\\
+ (-1)^{n-k+j}(q-1)\sum_{i\in\Z\setminus I(j-k)} (-1)^{\delta_{i<0}} T_{n+1-k-i}T_{n-j+i}\\
+ (-1)^{n-j+k}q(q-1)\sum_{i\in\Z\setminus I(-1+j-k)
}(-1)^{\delta_{i<0}} T_{n-k-i} T_{n+1-j+i}
 \Big) T_1T_n
\end{multline}
Applying
 \eqref{eq:T_identity_1} with $(j,k)$ replaced by $(n-j+1,n-k)$
and
 \eqref{eq:T_identity_2} with $(j,k)$ replaced by $(n-j,n-k+1)$
we obtain that the parenthesis in \eqref{eq:E0EnT1Tn_step} equals
\[(-1)^{n+j-k}qT_{n+1-k}T_{n-j} + (-1)^{n+k-j}qT_{n+1-j}T_{n-k}\]
which is symmetric in $j$ and $k$. This completes the proof that
\eqref{eq:big_sum_of_stuff} is symmetric in $j$ and $k$. Thus
\eqref{eq:tilde_T_commute} holds.

The last statement about generators follows from the fact that
\eqref{eq:tilde_T_def} and \eqref{eq:tilde_E_def} can be used to
express $E_j$ for $j\in\iv{1}{n-1}$ and $T_k$ for $k\in\iv{1}{n}$,
in terms of the new generators
$\{E_0,E_n\}\cup\{\widetilde{T}_j\}_{j=1}^{n-1}\cup\{\widetilde{E}_k\}_{k=0}^{n-1}$.

\subsection{Example: The case \texorpdfstring{$n=2$}{n=2}}
If $n=2$ then \eqref{eq:tsystem} becomes
\[y_1=t_1+x_1 t_2 ,\qquad y_2=t_1+x_2 t_2\]
and from this, or using \eqref{eq:formula_for_tj}, we get
\begin{align*}
t_1 &=(x_1-x_2)^{-1}(x_1y_2-x_2y_1),\\
t_2 &=-(x_1-x_2)^{-1}(y_2-y_1).
\end{align*}
By definition \eqref{eq:ed_def}, we have
\[e_0=1,\quad e_1=x_1+x_2, \quad e_2=x_1x_2. \]
By Corollary~\ref{cor:generators_of_skewfield_invariants},
$\K_q(\bar x,\bar y)^{S_2}$ is generated as a skew field over $\K$
by $e_1,e_2, t_1, t_2$. By Proposition~\ref{prp:tj_ek} we have the
following relations:
\begin{align*}
t_1 t_2 &= t_2 t_1, \\
e_1 e_2 &= e_2 e_1, \\
t_1 e_2 & = q e_2 t_1, \\
t_2 e_2 & = q e_2 t_2, \\
t_1 e_1 & = e_1 t_1 + (1-q)e_2 t_2, \\
t_2 e_1 & = q e_1 t_2 + (q-1)t_1.
\end{align*}
Using the notation in \eqref{eq:qNoether_XY_def} and
\eqref{eq:eki_def} we have
\begin{align*}
X_1 &= e_2^{(0)} = e_2, \\
X_2 &= e_1^{(1)} = t_2, \\
Y_1 &= e_0^{(1)} = t_1, \\
Y_2 &= e_0^{(2)} = e_1^{(0)}e_0^{(1)}e_1^{(1)}+e_0^{(0)}e_0^{(1)}e_0^{(1)}+e_2^{(0)}e_1^{(1)}e_1^{(1)}=\\
&= e_1t_1t_2+e_0t_1^2+e_2t_2^2.
\end{align*}
By \eqref{eq:qNoether_XY_relations} or direct computations,
\begin{alignat*}{3}
[Y_1,Y_2]&=0,&\qquad [X_2,X_1]_q &=0,\\
[Y_1,X_2]&=0,&\qquad [Y_2,X_1]_{q^2} &=0,\\
[Y_1,X_1]_q&=0,&\qquad [Y_2,X_2]_{q^{-1}} &=0.
\end{alignat*}
Thus, $(Z_1,Z_2,Z_3,Z_4)=(X_1,Y_1,X_2,Y_2)$ satisfy
$Z_iZ_j=q^{s_{ij}}Z_jZ_i$ with
\[
(s_{ij})=\begin{bmatrix}
0 & -1 & -1 & -2 \\
1 & 0  &  0 & 0 \\
1 & 0 &  0  & 1 \\
2 & 0 & -1 & 0
\end{bmatrix}.
\]
Using the definition \eqref{eq:qNoether_hatXhatY_def},
\[\widehat{X}_1=X_1,\quad \widehat{X}_2=Y_1X_2^{-1},
\quad \widehat{Y}_1=Y_1,\quad\widehat{Y}_2=Y_1^{-2}Y_2.\] By
Theorem~\ref{thm:qNoether_main_result},
 $\widehat{X}_1, \widehat{X}_2, \widehat{Y}_1,\widehat{Y}_2$ generate
 $\K(\bar x,\bar y)^{S_2}$ as a skew field and the following relations hold:
\begin{gather*}
 [\widehat{X}_1,\widehat{X}_2]=0,\qquad [\widehat{Y}_1,\widehat{Y}_2]=0,\\
 \widehat{Y}_i\widehat{X}_j = q^{\delta_{ij}} \widehat{X}_j\widehat{Y}_i,\qquad\forall i,j\in\{1,2\}.
\end{gather*}
This shows that $\K_q(x_1,x_2,y_1,y_2)^{S_2}\simeq
\K_q(x_1,x_2,y_1,y_2)$.

\subsection{Example: The case \texorpdfstring{$n=3$}{n=3}}
The elementary symmetric polynomials $e_d$ are
\begin{align*}
e_0 &= 1,\\
e_1 &=x_1+x_2+x_3,\\
e_2 &=x_1x_2+x_2x_3+x_3x_1,\\
e_3 &=x_1x_2x_3.
\end{align*}
By \eqref{eq:formula_for_tj} we have
\begin{align*}
t_1&=\Delta^{-1}\cdot\big( (x_2^2x_3-x_3^2x_2)y_1+(x_3^2x_1-x_1^2x_3)y_2+(x_1^2x_2-x_2^2x_1)y_3\big),\\
t_2&=\Delta^{-1}\cdot\big( (x_2^2-x_3^2)y_1+(x_3^2-x_1^2)y_2+(x_1^2-x_2^2)y_3\big),\\
t_3&=\Delta^{-1}\cdot\big(
(x_2-x_3)y_1+(x_3-x_1)y_2+(x_1-x_2)y_3\big),
\end{align*}
where
\[\Delta=(x_1-x_2)(x_1-x_3)(x_2-x_3).\]
By Corollary ~\ref{cor:generators_of_skewfield_invariants},
$\K_q(\bar x,\bar y)^{S_3}$ is generated as a skew field over $\K$
by $e_1,e_2,e_3,t_1,t_2,t_3$ and by Proposition \ref{prp:tj_ek} or
direct computations, we have the following relations:
\begin{align*}
[t_i,t_j]&=0, \quad\forall i,j\in\{1,2,3\},\\
[e_i,e_j]&=0,\quad\forall i,j\in\{1,2,3\},\\
[t_i,e_3]_q&=0,\quad\forall  i\in\{1,2,3\},\\
[t_1,e_1] &= (q-1)e_3 t_3,\\
[t_2,e_1] &= (q-1)(t_1 -e_2 t_3), \\
[t_3,e_1]_q &= (q-1)t_2,\\
[t_1,e_2] &= (1-q)e_3 t_2, \\
[t_2,e_2]_q &= (1-q)(e_3t_3-e_1 t_1),\\
[t_3,e_2]_q &= (1-q)t_1.
\end{align*}
By \eqref{eq:qNoether_XY_def} and \eqref{eq:eki_def},
\begin{align*}
X_1&=e_3^{(0)} = e_3=x_1x_2x_3,\\
X_2&=e_2^{(1)} = t_3,\\
X_3&=e_1^{(2)} = e_2^{(0)}e_0^{(1)}e_2^{(1)}-e_0^{(0)}e_0^{(1)}e_0^{(1)}+e_3^{(0)}e_1^{(1)}e_2^{(1)}=\\
&=e_2t_1t_3-e_0t_1^2+e_3t_2t_3,\\
Y_1&=e_0^{(1)}=t_1,\\
Y_2&=e_0^{(2)}=e_1^{(0)}e_0^{(1)}e_2^{(1)}+e_0^{(0)}e_1^{(1)}e_0^{(1)}-e_3^{(0)}e_2^{(1)}e_2^{(1)}=\\
&=e_1t_1t_3+e_0t_2t_1-e_3t_3^2,\\
Y_3&=e_0^{(3)}=e_1^{(1)}e_0^{(2)}e_1^{(2)}+e_0^{(1)}e_0^{(2)}e_0^{(2)}+e_2^{(1)}e_1^{(2)}e_1^{(2)}=\\
&= t_2Y_2X_3+t_1Y_2^2+t_3X_3^2.
\end{align*}
By \eqref{eq:qNoether_XY_relations},
\begin{alignat*}{3}
[Y_k,Y_i]&=0,&\quad\forall k,i\in\{1,2,&3\},\\
[X_2,X_1]_q&=0,&\quad [X_3,X_1]_{q^2}&=0,&\quad [X_3,X_2]_{q^{-1}}&=0,\\
[Y_1,X_2]&=0,&\quad [Y_1,X_3]&=0,&\quad [Y_2,X_3]&=0,\\
[Y_1,X_1]_q&=0,&\quad [Y_2,X_1]_{q^2}&=0,&\quad [Y_3,X_1]_{q^5}&=0,\\
[Y_2,X_2]_{q^{-1}} &=0,&\quad [Y_3,X_2]_{q^{-2}}&=0,&\quad
[Y_3,X_3]_q&=0.
\end{alignat*}
Thus, if we let $(Z_1,Z_2,\ldots,Z_6)=(X_1,Y_1,X_2,Y_2,X_3,Y_3)$,
then $Z_iZ_j=q^{s_{ij}}Z_jZ_i$ with
\begin{equation}
(s_{ij})=\begin{bmatrix}
0 &-1 &-1 &-2 &-2 &-5 \\
1 & 0 & 0 & 0 & 0 & 0 \\
1 & 0 & 0 & 1 & 1 & 2 \\
2 & 0 &-1 & 0 & 0 & 0 \\
2 & 0 &-1 & 0 & 0 &-1 \\
5 & 0 &-2 & 0 & 1 & 0
\end{bmatrix}.
\end{equation}
By performing simultaneous elementary row and column
transformations, this matrix can be brought to the skew normal
form
\begin{equation}
\begin{bmatrix}
0 & 1 & 0 & 0 & 0 & 0 \\
-1& 0 & 0 & 0 & 0 & 0 \\
0 & 0 & 0 &-1 & 0 & 0 \\
0 & 0 & 1 & 0 & 0 & 0 \\
0 & 0 & 0 & 0 & 0 & 1 \\
0 & 0 & 0 & 0 &-1 & 0
\end{bmatrix}.
\end{equation}
As in \eqref{eq:qNoether_hatXhatY_def}, changing generators to
\begin{gather*}
\widehat{X}_1=X_1,\quad \widehat{X}_2=Y_1X_2^{-1},\quad \widehat{X}_3=Y_2^{-1}X_3,\\
\widehat{Y}_1=Y_1,\quad \widehat{Y}_2=Y_1^{-2}Y_2,\quad
\widehat{Y}_3=Y_1^{-1}Y_2^{-2}Y_3.
\end{gather*}
one can also verify directly that
\begin{gather*}
[\widehat{X}_i,\widehat{X}_j]=[\widehat{Y}_i,\widehat{Y}_j]=0,\quad\forall i,j\in\{1,2,3\},\\
\widehat{Y}_i\widehat{X}_j = q^{\delta_{ij}} \widehat{X}_j
\widehat{Y}_i,\quad\forall i,j\in\{1,2,3\},
\end{gather*}
which means that there is an isomorphism of skew fields
\begin{align*}
 \K_q(\bar x,\bar y) &\overset{\sim}{\longrightarrow} \K_q(\bar x,\bar y)^{S_3}\\
 x_i &\longmapsto \widehat{X}_i,\qquad\forall i\in\{1,2,3\},\\
 y_i &\longmapsto \widehat{Y}_i,\qquad\forall i\in\{1,2,3\}.\\
\end{align*}

\section{Acknowledgment}
The first author is grateful to the Max Planck Institute for
Mathematics in Bonn for support and hospitality during his visit.
The first author is supported in part by the CNPq grant
(301743/2007-0) and by the Fapesp grant (2010/50347-9).

The  authors are grateful to Michel Van den Bergh, Fedor Malikov,
Eugene Mukhin and Alan Weinstein for encouraging discussions.

\end{document}